\documentclass[12pt]{amsart}
\usepackage{a4wide,enumerate,color}
\usepackage{amsmath}
\usepackage{scrextend} 
\usepackage{graphicx}
\usepackage{caption}
\allowdisplaybreaks

\let\pa\partial  
\let\na\nabla  
\let\eps\varepsilon  
\newcommand{\N}{{\mathbb N}}  
\newcommand{\R}{{\mathbb R}} 
\newcommand{\diver}{\operatorname{div}}

\newcommand{\M}{{\mathcal M}}
\newcommand{\T}{{\mathcal T}}
\newcommand{\E}{{\mathcal E}}
\newcommand{\Pt}{{\mathcal P}}
\newcommand{\DD}{{\mathcal D}}
\newcommand{\F}{{\mathcal F}}
\newcommand{\dist}{{\mathrm d}}
\newcommand{\m}{{\mathrm m}}
\newcommand{\dom}{{\mathcal O}}

\newtheorem{thrm}{Theorem}[section]
\newtheorem{prpstn}{Proposition}[section]

\newtheorem{lem}{Lemma}[section]
\newtheorem{rmrk}{Remark}[section]

\begin{document}  

\title[Biofilm cross-diffusion system]{Convergence of a finite-volume scheme for
a degenerate-singular cross-diffusion system \\ for biofilms}

\author[E. S. Daus]{Esther S. Daus}
\address{Institute for Analysis and Scientific Computing, Vienna University of  
	Technology, Wiedner Hauptstra\ss e 8--10, 1040 Wien, Austria}
\email{esther.daus@tuwien.ac.at} 

\author[A. J\"ungel]{Ansgar J\"ungel}
\address{Institute for Analysis and Scientific Computing, Vienna University of  
	Technology, Wiedner Hauptstra\ss e 8--10, 1040 Wien, Austria}
\email{juengel@tuwien.ac.at} 

\author[A. Zurek]{Antoine Zurek}
\address{Institute for Analysis and Scientific Computing, Vienna University of  
	Technology, Wiedner Hauptstra\ss e 8--10, 1040 Wien, Austria}
\email{antoine.zurek@tuwien.ac.at} 

\date{\today}

\thanks{The authors acknowledge partial support from
the French-Austrian Amad\'ee project of the OeAD Austria. 
The first and second authors have been supported by the Austrian Science Fund (FWF), 
grants P33010, P30000, W1245, and F65.} 

\begin{abstract}
An implicit Euler finite-volume scheme for a cross-diffusion system modeling 
biofilm growth is analyzed by exploiting its formal gradient-flow structure. 
The numerical scheme is based on a two-point flux approximation that preserves the 
entropy structure of the continuous model. Assuming equal diffusivities, 
the existence of nonnegative and bounded 
solutions to the scheme and its convergence are proved. Finally, we supplement 
the study by numerical experiments in one and two space dimensions.
\end{abstract}

 

\maketitle


\section{Introduction}

Biofilms are organized, cooperating communities of microorganisms. They can be used
for the treatment of wastewater \cite{CaRo92,NVH00}, as they help to reduce sulfate and
to remove nitrogen. Typically, biofilms consist of several species
such that multicomponent fluid models need to be considered. 
Recently, a multi-species biofilm model was introduced by Rahman, Sudarsan, and
Eberl \cite{RSE15}, which reflects the same properties as the single-species
diffusion model of \cite{EPV01}. The model has a porous-medium-type
degeneracy when the local biomass vanishes, and a singularity when the biomass 
reaches the maximum capacity, which guarantees the boundedness of the total mass.
The model was derived formally from a space-time discrete walk
on a lattice in \cite{RSE15}. The global existence of weak solutions to the
single-species model was proved in \cite{EZE09}, while the global existence
analysis for the multi-species cross-diffusion system can be found in \cite{DMZ18}.
The proof of the multi-species model is based on an entropy method which also
provides the boundedness of the biomass hidden in its entropy structure.
Numerical simulations were performed in \cite{DMZ18,RSE15}, but no numerical
analysis was given. In this paper, we analyze an implicit Euler finite-volume
scheme of the multi-species system that preserves the structure of the
continuous model, namely positivity, boundedness, and discrete entropy production.

The model equations for the proportions of the biofilm species $u_i$ are given by
\begin{equation}\label{1.eq}
  \pa_t u_i + \diver \F_i = 0, \quad 
	\F_i = -\alpha_i p(M)^2\na\frac{u_iq(M)}{p(M)}
	\quad\mbox{in }\Omega,\ t>0,\ i=1,\ldots,n,
\end{equation}
where $\Omega\subset\R^d$ ($d\geq 1$) is a bounded domain, 
$\alpha_i>0$ are some diffusion coefficients, 
and $M=\sum_{i=1}^n u_i$ is the total biomass. 
The proportions $u_i(x,t)$ are nonnegative and satisfy $M\leq 1$.
We have assumed for simplicity that the functions $p$ and $q$ only depend on the
total biomass and are the same for all species. 
The function $p\in C^1([0,1])$ is decreasing and satisfies $p(1)=0$, and
$q$ is defined by 
\begin{equation}\label{1.q}
  q(M) := \frac{p(M)}{M}\int_0^M\frac{s^a}{(1-s)^b}\frac{ds}{p(s)^2}, \quad M>0,
\end{equation}
where $a$, $b\geq 1$.
Equations \eqref{1.eq} are complemented by initial and mixed boundary conditions:
\begin{align}
  & u_i(0) = u_i^0\quad\mbox{in }\Omega,\ i=1,\ldots,n, \label{1.ic} \\
	& u_i = u_i^D\quad\mbox{on }\Gamma^D, \quad \na \F_i\cdot\nu = 0\quad\mbox{on }
	\Gamma^N, \label{1.bc}
\end{align}
where $\Gamma^D$ is the contact boundary part, $\Gamma^N$ is the union
of isolating boundary parts, and $\pa\Omega=\Gamma^D\cup\Gamma^N$.

We recover the single-species model if all species are the same and 
all diffusivities $\alpha_i$ are equal, $\alpha_i=1$ for $i=1,\ldots,n$. Indeed, 
summing \eqref{1.eq} over $i=1,\ldots,n$, it follows that
\begin{align}\label{1.eqM}
  \pa_t M = \diver\bigg(p(M)^2\na\frac{Mq(M)}{p(M)}\bigg) 
	= \diver\bigg(\frac{M^a}{(1-M)^b}\na M\bigg),
\end{align}
which makes the degenerate-singular structure of the model evident.

Equations \eqref{1.eq} can be written as the cross-diffusion system
\begin{equation}\label{1.eq2}
  \pa_t u_i - \diver\bigg(\sum_{j=1}^n A_{ij}(u)\na u_j\bigg) = 0
	\quad\mbox{in }\Omega,\ t>0,
\end{equation}
where the nonlinear diffusion coefficients are defined by
\begin{equation}\label{1.A}
  A_{ij}(u) = \alpha_i\delta_{ij}p(M)q(M) + \alpha_i u_i\big(p(M)q'(M)-p'(M)q(M)\big),
	\quad i,j=1,\ldots,n.
\end{equation}
Due to the cross-diffusion structure, standard techniques like the maximum principle
and regularity theory cannot be used. Moreover, the diffusion matrix
$(A_{ij}(u))$ is generally neither symmetric nor positive definite.

The key of the analysis, already observed in \cite{DMZ18}, is that system 
\eqref{1.eq2}-\eqref{1.A} allows for an entropy or formal gradient-flow structure.
Indeed, introduce the (relative) entropy
\begin{align*}
  & H(u) = \int_\Omega h^*(u|u^D)dx, \quad\mbox{where} \\
	& h^*(u|u^D) = h(u) - h(u^D) - h'(u^D)\cdot(u-u^D), \\
  & h(u) = \sum_{i=1}^n\big(u_i(\log u_i-1)+1\big)
	+ \int_0^M\log\frac{q(s)}{p(s)}ds,
\end{align*}
defined on the set 
\begin{equation}\label{1.O}
  \dom = \bigg\{u=(u_1,\ldots,u_n)\in(0,\infty)^n: \sum_{i=1}^n u_i<1\bigg\}.
\end{equation}
A computation gives the entropy identity \cite[Theorem 2.1]{DMZ18}
\begin{equation}\label{1.edi}
  \frac{dH}{dt} + 2\sum_{i=1}^n\alpha_i\int_\Omega p(M)^2\bigg|\na
	\sqrt{\frac{u_iq(M)}{p(M)}}\bigg|^2 dx = 0.
\end{equation}
Thus, $H$ is a Lyapunov functional along the solutions to \eqref{1.eq}.
Moreover, under some assumptions on $p$, the entropy production term 
(the second term on the left-hand side)
can be bounded from below, for some constant $C>0$, by
\begin{align}
  \sum_{i=1}^n&\alpha_i\int_\Omega p(M)^2\bigg|\na
	\sqrt{\frac{u_iq(M)}{p(M)}}\bigg|^2 dx \nonumber \\
  &\ge C\int_\Omega\frac{M^{a-1}|\na M|^2}{(1-M)^{1+b+\kappa}} dx
	+ \sum_{i=1}^n\int_\Omega p(M)q(M)|\na\sqrt{u_i}|^2 dx, \label{1.ineq}
\end{align}
yielding suitable gradient estimates. Moreover, it implies that
$(1-M)^{1-b-\kappa}$ is integrable, showing that $M<1$ a.e.\ in $\Omega$, $t>0$,
which excludes biofilm saturation and allows us to define the nonlinear terms.

Another feature of the entropy method is that equations \eqref{1.eq}, written in
the so-called entropy variables $w_i=\pa h^*/\pa u_i$, can be written as the
formal gradient-flow system
$$
  \pa_t u - \diver(B(w)\na w) = 0,
$$
with a positive semidefinite diffusion matrix $B$. Since the derivative
$(h^*)':\dom\to\R^n$ is invertible \cite[Lemma 3.3]{DMZ18}, 
$u$ can be interpreted as a function of $w$, $u(w)=[(h^*)']^{-1}(w)$,
mapping $\R^n$ to $\dom$. This gives automatically $u(w)\in\dom$
and consequently $L^\infty$ bounds.
This property, for another volume-filling model, was first observed in
\cite{BDPS10} and later generalized in \cite{Jue15}. 

The aim of this paper is to reproduce the above-mentioned properties 
on the discrete level.
For this, we suggest an implicit Euler scheme in time (with time step size
$\Delta t$) and a finite-volume discretization in space
(with grid size parameter $\Delta x$), based on two-point approximations. 
The challenge is to formulate the discrete fluxes such that the scheme preserves the entropy structure of the model and to design the fluxes such that we are able 
to establish the upper bound $M < 1$ a.e.\ in $\Omega$, $t>0$. We suggest the
discrete fluxes \eqref{sch2}, where the coefficient $p(M)^2$ is replaced by
$(p(M_K)^2+p(M_L)^2)/2$, and $K$ and $L$ are two neighboring control volumes
with a common edge (see Section \ref{sec.nota} for details). We establish a discrete counterpart of~\eqref{1.ineq} in Lemma~\ref{lem.lower}. This result is proved by exploiting the properties
of the functions $p$ and $q$ as in \cite[Lemma 3.4]{DMZ18} and distinguishing carefully
the cases $M\le 1-\delta$ and $M>1-\delta$ for sufficiently small $\delta>0$. However, due to the lack of a chain rule at the discrete level, we cannot conclude that the ``discrete'' biomass satisfies $M < 1$. To overcome this issue, we need to assume that the diffusivities are all equal. Then, summing the finite-volume analog of
\eqref{1.eq} over $i=1,\ldots,n$, 
we obtain a discrete analog of the diffusion equation \eqref{1.eqM} for $M$ 
that allows us to apply a discrete maximum principle, leading to $M<1$. 

Our results can be sketched as follows 
(see Section \ref{sec.main} for the precise statements):
\begin{itemize}
\item[(i)] We prove the existence of finite-volume solutions with 
nonnegative discrete proportions $u_{i,K}$ and discrete total
biomass $M_K<1$ for all control volumes $K$.
\item[(ii)] The discrete solution satisfies a discrete analog of the entropy equality
(which becomes an inequality in \eqref{2.edi}) and of the lower bound
\eqref{1.ineq} for the entropy production.
\item[(iii)] The discrete solution converges in a certain sense, 
for mesh sizes $(\Delta x,\Delta t)\to 0$,
to a weak solution to \eqref{1.eq}.
\end{itemize}
Let us notice that even if the assumption on the diffusion coefficients provides 
an upper bound for $M$, we cannot establish the nonnegativity of the densities 
$u_i$ by using a maximum principle. Instead, we adapt at the discrete level the so-called boundedness-by-entropy method, introduced in~\cite{BDPS10} and developed 
in~\cite{Jue15}, to a finite-volume scheme. This approach allows us to prove 
that the solutions to the nonlinear scheme proposed in this paper satisfy 
the properties~(i)-(iii); see Theorems~\ref{thm.ex} and \ref{thm.conv}. 
The adaptation of this technique represents the main originality of this work.

There are several finite-volume schemes for other cross-diffusion systems 
in the mathematical literature. For instance, an upwind two-point flux
approximation was used in \cite{Ait18} for a seawater intrusion model.
A positivity-preserving two-point flux approximation for a two-species
population system was suggested in \cite{ABB11}. The Laplacian structure of
the population model was exploited in \cite{Mur17} to design a convergent linear
finite-volume scheme, avoiding fully implicit approximations. Cross-diffusion
systems with nonlocal (in space) terms modeling food chains and epidemics
were approximated in \cite{ABLS15,ABS15}. The convergence of the finite-volume
scheme of a degenerate cross-diffusion system arising in ion transport was
shown in \cite{CCGJ19}, and the existence of a finite-volume
scheme for a population cross-diffusion system was proved in \cite{JuZu19}.

A finite-volume scheme for the biofilm growth, coupled with the computation
of the surrounding fluid flow, was presented in \cite{YaUe13}.
Finite-volume-based simulations of biofilm processes in axisymmetric reactors were
given in \cite{SCC93}. Closer to our numerical study is the work \cite{RaEb14},
where the single-species biofilm model was discretized using finite volumes,
but without any numerical analysis. In this paper, we prove the existence of
discrete solutions and the convergence of the finite-volume scheme for 
\eqref{1.eq} for the first time.

The paper is organized as follows. The notation and assumptions on the mesh
as well as the main theorems are introduced in Section \ref{sec.num}.
The existence of discrete solutions is proved in Section \ref{sec.ex},
based on a topological degree argument. We show a gradient estimate, an estimate of
the discrete time derivative, and the lower bound for the entropy production
in Section \ref{sec.apriori}. These estimates allow us in Section \ref{sec.convsol}
to apply the discrete compactness argument in \cite{ACM17} 
to conclude the a.e.\ convergence of the proportions and to show the
convergence of the discrete gradient associated to $\na(u_iq(M)/p(M))$.
The convergence of the scheme is then proved in Section \ref{sec.convsch}. 
In Section \ref{sec.exp.num}, we present some numerical results in one and two 
space dimensions. They illustrate the $L^2$-convergence rate in space of the 
numerical scheme and show the convergence of the solutions to the steady states.


\section{Numerical scheme and main results}\label{sec.num}

In this section, we introduce the numerical scheme and detail our main results.

\subsection{Notation and assumptions}\label{sec.nota}

Let $\Omega\subset\R^2$ be an open, bounded, 
polygonal domain with $\pa\Omega=\Gamma^D\cup\Gamma^N\in C^{0,1}$, 
$\Gamma^D\cap\Gamma^N=\emptyset$, and $\operatorname{meas}(\Gamma^D)>0$.
We consider only two-dimensional domains $\Omega$, but the generalization to
higher dimensions is straightforward.
An admissible mesh $\M=(\T,\E,\Pt)$ of $\Omega$ is given by a family $\T$ of open
polygonal control volumes (or cells), 
a family $\E$ of edges, and a family $\Pt$ of points
$(x_K)_{K\in\T}$ associated to the control volumes and satisfying Definition 
9.1 in \cite{EGH00}. This definition implies that the straight line between two 
centers of neighboring cells $\overline{x_Kx_L}$ 
is orthogonal to the edge $\sigma=K|L$ between two
cells $K$ and $L$. The condition is satisfied by, for instance, triangular meshes
whose triangles have angles smaller than $\pi/2$ \cite[Examples 9.1]{EGH00}
or Vorono\"i meshes \cite[Example 9.2]{EGH00}.

The family of edges $\E$ is assumed to consist of the interior edges 
$\sigma\in\E_{\rm int}$ satisfying $\sigma\in\Omega$ and the boundary edges 
$\sigma\in\E_{\rm ext}$ fulfilling $\sigma\subset\pa\Omega$.
We suppose that each exterior edge is an element of
either the Dirichlet or Neumann boundary, i.e.\ $\E_{\rm ext}=\E_{\rm ext}^D\cup
\E_{\rm ext}^N$. For a given control volume $K\in\T$, we denote by $\E_K$ the
set of its edges. This set splits into $\E_K=\E_{{\rm int},K}\cup
\E_{{\rm ext},K}^D\cup\E_{{\rm ext},K}^N$. 
For any $\sigma\in\E$, there exists at least one cell
$K\in\T$ such that $\sigma\in\E_K$. We denote this cell by $K_\sigma$.
When $\sigma$ is an interior cell, $\sigma=K|L$, $K_\sigma$ can be either $K$
or $L$. 

Let $\sigma\in\E$ be an edge. We define
$$
  \dist_\sigma = \left\{\begin{array}{ll}
	\dist(x_K,x_L) &\quad\mbox{if }\sigma=K|L\in\E_{{\rm int}}, \\
	\dist(x_K,\sigma) &\quad\mbox{if }\sigma\in\E_{{\rm ext},K},
	\end{array}\right.
$$
where d is the Euclidean distance in $\R^2$. 
The transmissibility coefficient is defined by
\begin{equation}\label{2.trans}
  \tau_\sigma = \frac{\m(\sigma)}{\dist_\sigma},
\end{equation}
where $\m(\sigma)$ denotes the Lebesgue measure of $\sigma$. We assume that
the mesh satisfies the following regularity requirement: There exists $\xi>0$
such that 
\begin{equation}\label{2.regmesh}
  \dist(x_K,\sigma) \ge \xi\dist_\sigma\quad\mbox{for all }K\in\T,\ \sigma\in\E_K.
\end{equation}
This hypothesis is needed to apply a discrete Sobolev inequality; see
\cite{BCF15}.

The size of the mesh is denoted by $\Delta x=\max_{K\in\T}\operatorname{diam}(K)$.
Let $N_T\in\N$ be the number of time steps, $\Delta t=T/N_T$ be the
time step and set $t_k=k\Delta t$ for $k=0,\ldots,N_T$.
We denote by $\DD$ an admissible space-time discretization of
$Q_T:=\Omega\times(0,T)$ composed of an admissible mesh $\M$ of $\Omega$
and the values $(\Delta t,N_T)$. The size of $\DD$ is defined by
$\eta:=\max\{\Delta x,\Delta t\}$. 

As it is usual for the finite-volume method, we introduce functions that are
piecewise constant in space and time. A finite-volume scheme provides a vector
$v_\T=(v_K)_{K\in\T}\in\R^{\#\T}$ of approximate values of a function $v$
and the associate piecewise constant function, still denoted by $v_\T$,
$$
  v_\T = \sum_{K\in\T}v_K\mathbf{1}_K,
$$
where $\mathbf{1}_K$ is the characteristic function of $K$. The vector
$v_\M$, containing the approximate values in the control volumes and the
approximate values on the Dirichlet boundary edges, is written as 
$v_\M=(v_\T,v_{\E^D})$, where $v_{\E^D}=(v_\sigma)_{\sigma\in\E_{{\rm ext}}^D}
\in\R^{\#\E_{\rm ext}^D}$. For a vector $v_\M$, we introduce for $K\in\T$
and $\sigma\in\E_K$ the notation
\begin{equation}\label{2.vKsigma}
  v_{K,\sigma} = \left\{\begin{array}{ll}
	v_L &\quad\mbox{if }\sigma=K|L\in\E_{{\rm int},K}, \\
	v_\sigma &\quad\mbox{if }\sigma\in\E_{{\rm ext},K}^D, \\
	v_K &\quad\mbox{if }\sigma\in\E_{{\rm ext},K}^N
	\end{array}\right.
\end{equation}
and the discrete gradient
\begin{equation}\label{2.Dsigma}
  D_\sigma v := |D_{K,\sigma}v|, \quad\mbox{where }D_{K,\sigma}v = v_{K,\sigma}-v_K.
\end{equation}
The discrete $H^1(\Omega)$ seminorm and the (squared) discrete $H^1(\Omega)$ norm are then
defined by
\begin{align}\label{2.norm}
  |v_\M|_{1,2,\M} = \bigg(\sum_{\sigma\in\E}\tau_\sigma(D_\sigma v)^2\bigg)^{1/2},
	\quad \|v_\M\|^2_{1,2,\M} = \|v_\M\|^2_{0,2,\M} + |v_\M|^2_{1,2,\M},
\end{align}
where $\|\cdot\|_{0,p,\M}$ denotes the $L^p(\Omega)$ norm
$$
  \|v_\M\|_{0,p,\M} = \bigg(\sum_{K\in\T}\m(K)|v_K|^p\bigg)^{1/p}, \quad \forall 1 \leq p < \infty.
$$

Thanks to the regularity assumption \eqref{2.regmesh} and the fact that
$\Omega$ is two-dimensional, we have
\begin{equation}\label{2.estmesh}
  \sum_{K\in\T}\sum_{\sigma\in\E_K}\m(\sigma)\dist(x_K,\sigma)
	\le 2\sum_{K\in\T}\m(K) = 2\m(\Omega).
\end{equation}


\subsection{Numerical scheme}

We are now in the position to define the finite-volume discretization of 
\eqref{1.eq}-\eqref{1.bc}. Let $\DD$ be a finite-volume discretization of $Q_T$.
The initial and boundary conditions are discretized by the averages
\begin{align}
  u_{i,K}^0 &= \frac{1}{\m(K)}\int_K u_i^0(x)dx \quad\mbox{for }K\in\T,
  \label{sch.ic} \\
  u_{i,\sigma}^D &= \frac{1}{\m(\sigma)}\int_\sigma u_i^Dds \quad\mbox{for }
	\sigma\in\E_{\rm ext}^D,\ i=1,\ldots,n. \label{sch.bc}
\end{align}
We suppose for simplicity that the Dirichlet datum is constant on $\Gamma^D$
such that $u_{i,\sigma}^D=u_i^D$ for $i=1,\ldots,n$. Furthermore, we set
$u_{i,\sigma}^k=u_{i,\sigma}^D$ for $\sigma\in\E_{\rm ext}^D$ at time $t_k$.

Let $u_{i,K}^k$ be an approximation of the mean value of $u_i(\cdot,t_k)$
in the cell $K$. Then the implicit Euler finite-volume scheme reads as
\begin{align}
  & \frac{\m(K)}{\Delta t}(u_{i,K}^{k}-u_{i,K}^{k-1})
	+ \sum_{\sigma\in\E_K}\F_{i,K,\sigma}^{k} = 0, 
	\label{sch1} \\
	& \F_{i,K,\sigma}^{k} = -\tau_\sigma\alpha_i(p_\sigma^{k})^2
	D_{K,\sigma}\bigg(\frac{u_i^{k} q(M^{k})}{p(M^{k})}\bigg), \label{sch2}
\end{align}
where $K\in\T$, $\sigma\in\E_K$, $i=1,\ldots,n$, and the value $p_\sigma^{k}$
is defined by
\begin{equation}\label{2.psigma}
  (p_\sigma^{k})^2 := \frac{p(M_K^{k})^2+p(M_{K,\sigma}^{k})^2}{2},
\end{equation}
recalling definition \eqref{2.trans} for $\tau_\sigma$ and notation 
\eqref{2.vKsigma} for $M^k_{K,\sigma}$. 

Observe that definitions \eqref{2.vKsigma} and \eqref{2.Dsigma} ensure that
the discrete fluxes vanish on the Neumann boundary edges, i.e.\
$\F_{i,K,\sigma}^{k}=0$ for all $\sigma\in\E_{{\rm ext},K}^N$, $k\in\N$,
and $i=1,\ldots,n$. This is consistent with the Neumann boundary conditions
in \eqref{1.bc}. 

For the convergence result, we need to define the discrete gradients.
To this end, let the vector $u_\M=(u_\T,u_{\E^D})$ as defined before. Then we introduce
the piecewise constant approximation $u_\DD = (u_{1,\DD},\ldots,u_{n,\DD})$ by
\begin{align}\label{reconstruc}
  u_{i,\DD}(x,t) = \sum_{K\in\T}u_{i,K}^k\mathbf{1}_K(x)
	&\quad\mbox{for }x\in\Omega,\ t\in(t_{k-1},t_k], \\
	\label{reconstrucbord}
	u_{i,\DD}(x,t) = u_i^D &\quad\mbox{for }x\in\Gamma^D, i=1,\ldots,n.
\end{align}
For given $K\in\T$ and $\sigma\in\E_K$,
we define the cell $T_{K,\sigma}$ of the dual mesh as follows:
\begin{itemize}
\item If $\sigma=K|L\in\E_{{\rm int},K}$, then $T_{K,\sigma}$ is that cell 
(``diamond'') whose
vertices are given by $x_K$, $x_L$, and the end points of the edge $\sigma$.
\item If $\sigma\in\E_{{\rm ext},K}$, then $T_{K,\sigma}$ is that cell (``triangle'')
whose vertices are given by $x_K$ and the end points of the edge $\sigma$.
\end{itemize}
An example of a construction of such a dual mesh can be found in \cite{CLP03}.
The cells $T_{K,\sigma}$ define a partition of $\Omega$. The definition of the 
dual mesh implies the following properties:
\begin{itemize}
\item 
As the straight line between two neighboring centers of cells $\overline{x_Kx_L}$ is
orthogonal to the edge $\sigma=K|L$, it follows that
\begin{equation}\label{2.para}
  \m(\sigma)\dist(x_K,x_L) = 2\m(T_{K,\sigma}) \quad\mbox{for all }
	\sigma=K|L\in \E_{{\rm int},K}.
\end{equation}
\item The property $\m(T_{K,\sigma})=\m(T_{L,\sigma})$ for
$\sigma=K|L\in\E_{{\rm int},K}$ implies that
\begin{equation*}
  \sum_{\substack{\sigma\in\E \\ K=K_\sigma}}\m(T_{K,\sigma}) \le 2\m(\Omega),
\end{equation*}
where the sum is over all edges $\sigma\in\E$, and to each given $\sigma$ 
we associate the cell $K=K_\sigma$.
\end{itemize}
We define the approximate gradient
of a piecewise constant function $u_\DD$ in $Q_T$ given by~\eqref{reconstruc}-\eqref{reconstrucbord} as follows:
$$
  \na^\DD u_{\DD}(x,t) = \frac{\m(\sigma)}{\m(T_{K,\sigma})}D_{K,\sigma} u^k \, \nu_{K,\sigma}
	\quad\mbox{for }x\in T_{K,\sigma},\ t\in(t_{k-1},t_k],
$$
where the discrete operator $D_{K,\sigma}$ is given in \eqref{2.Dsigma} and $\nu_{K,\sigma}$ is
the unit vector that is normal to $\sigma$ and points outward of $K$.


\subsection{Main results}\label{sec.main}

Our first result guarantees that scheme \eqref{sch.ic}-\eqref{2.psigma} possesses
a solution and that it preserves the entropy dissipation property.
Let us collect our assumptions:

\begin{labeling}{(H4)}
\item[(H1)] Domain: $\Omega\subset\R^2$ is a bounded
polygonal domain with Lipschitz boundary $\pa\Omega=\Gamma^D\cup\Gamma^N$,
$\Gamma^D\cap\Gamma^N=\emptyset$, and $\operatorname{meas}(\pa\Gamma^D)>0$.

\item[(H2)] Discretization: $\DD$ is an admissible discretization of $Q_T$
satisfying the regularity condition \eqref{2.regmesh}.

\item[(H3)] Data: $u^0=(u_1^0,\ldots,u_n^0)\in L^2(\Omega;[0,\infty)^n)$ ,
$u^D=(u_1^D,\ldots,u_n^D)\in(0,\infty)^n$ is a constant vector, 
$\sum_{i=1}^n u_i^0<1$ in $\Omega$, $\sum_{i=1}^n u^D_i<1$, and
$\alpha_1,\ldots,\alpha_n>0$, $a$, $b\geq 1$.

\item[(H4)] Functions: $p\in C^1([0,1];[0,\infty))$ is decreasing, $p(1)=0$,
and there exist $c$, $\kappa>0$ such that $\lim_{M\to 1}(-(1-M)^{1+\kappa}p'(M)/p(M))=c$.
The function $q$ is defined in \eqref{1.q}. 
\end{labeling}

For our main results, we need the following technical assumption:

\begin{labeling}{(A1)}
\item[(A1)] The diffusion constants are equal, $\alpha_i = 1$ for $i=1,\ldots,n$.
\end{labeling}

\begin{rmrk}[Discussion of the hypotheses]\rm 
	
The assumption on the behavior of $p$ when $M\to 1$ quantifies how fast this
function decreases to zero as $M\to 1$. An integration implies the bound
$$
  p(M) \le K_1\exp(-K_2(1-M)^{-\kappa}) \quad\mbox{for }0<M<1,
$$
with $K_1$ and $K_2$ some positive constants. We imposed this technical assumption to show a discrete version of~\eqref{1.ineq},
following the proof of \cite[Lemma 3.4]{DMZ18}; see Lemma~\ref{lem.lower}. The lower 
bound on the entropy production term is needed to prove the convergence result.

The upper bound for $p$ is also used in \cite{DMZ18} to deduce an estimate for
$(1-M)^{1-b-\kappa}$ in $L^1(\Omega)$, impliying that $M<1$ in $\Omega$.
Unfortunately, this estimate requires the multiple use of the chain rule
which is not available on the discrete level. Therefore, we assume that
the diffusivities $\alpha_i$ are equal and apply a weak maximum principle 
to the equation for $M^k$ to deduce the bound $M_K^k<1$ for all $K\in\T$.

In \cite{DMZ18}, the parameters in the definition \eqref{1.q} of $q$ need to satisfy
$a$, $b>1$. We are able to allow for the slightly weaker condition $a$, $b\ge 1$;
this is possible since we allow for equal diffusivities (condition (A1)).
\qed
\end{rmrk}

We introduce the discrete entropy
\begin{equation*} 
  H(u_\M^k) = \sum_{K\in\T}\m(K)h^*(u_K^k|u^D),
\end{equation*}
where 
\begin{align*}
  & h^*(u_K^k|u^D) = h(u_K^k) - h(u^D) - h'(u^D)\cdot(u_K^k-u^D) \\
  & \mbox{with }h(u_K^k) = \sum_{i=1}^n\big(u_{i,K}^k(\log u_{i,K}^k-1)+1\big)
	+ \int_0^{M_K^k}\log\frac{q(s)}{p(s)}ds
\end{align*}
is the relative entropy density.

\begin{thrm}[Existence of discrete solutions]\label{thm.ex}
Let hypotheses (H1)-(H4) and (A1) hold. Then there exists a solution 
$(u_K^k)_{K\in\T,\,k=0,\ldots,N_T}$ with $u_K^k=(u_{1,K}^k,\ldots,u_{n,K}^k)$ 
to scheme \eqref{sch.ic}-\eqref{2.psigma} satisfying
$$
  u_{i,K}^k\ge 0, \quad M_K^k = \sum_{i=1}^n u_{i,K}^k \leq M^* \quad\mbox{for }
	K\in\T,\ k=0,\ldots,N_T,
$$
where $M^* = \sup_{x \in \Omega} \lbrace M^D, M^0(x) \rbrace < 1$. Moreover, the discrete entropy dissipation inequality
\begin{equation}\label{2.edi}
  H(u_\M^{k}) + \Delta t\sum_{i=1}^n I_i(u_\M^{k}) 
	\le H(u_\M^{k-1}), \quad k=1,\ldots,N_T,
\end{equation}
holds with the entropy dissipation
\begin{equation}\label{2.ed}
  I_i(u_\M^{k}) = \sum_{\sigma\in\E }\tau_\sigma
	(p_\sigma^{k})^2\bigg(D_\sigma\bigg(\sqrt{\frac{u_i^{k}q(M^k)}{p(M^k)}}\bigg)\bigg)^2, \quad i=1,\ldots,n.
\end{equation}
\end{thrm}

For the convergence result, we introduce a family $(\DD_\eta)_{\eta>0}$ of
admissible space-time discretizations of $Q_T$ indexed by the size 
$\eta=\max\{\Delta x,\Delta t\}$ of the mesh. We denote by
$(\M_\eta)_{\eta>0}$ the corresponding meshes of $\Omega$. For any $\eta>0$,
let $u_\eta:=u_{\DD_\eta}$ be the finite-volume solution constructed in
Theorem \ref{thm.ex} and set $\na^\eta:=\na^{\DD_\eta}$.

\begin{thrm}\label{thm.conv}
Let the hypotheses of Theorem \ref{thm.ex} hold. 
Let $(\DD_\eta)_{\eta>0}$
be a family of admissible discretizations satisfying \eqref{2.regmesh} uniformly
in $\eta$. Furthermore, let $(u_\eta)_{\eta>0}$ be a family of finite-volume
solutions to scheme \eqref{sch.ic}-\eqref{2.psigma}. Then there exists a function
$u=(u_1,\ldots,u_n)$ satisfying $u(x,t)\in\overline{\dom}$ (see \eqref{1.O})
such that
\begin{align*}
  u_{i,\eta}\to u_i &\quad\mbox{a.e. in }Q_T,\ i=1,\ldots,n, \\
	M_\eta = \sum_{i=1}^n u_{i,\eta} \to M = \sum_{i=1}^n u_i<1 
	&\quad\mbox{a.e. in }Q_T, \\
	\na^\eta\bigg(\frac{u_{i,\eta}q(M_\eta)}{p(M_\eta)}\bigg) \rightharpoonup
	\na\bigg(\frac{u_iq(M)}{p(M)}\bigg) &\quad\mbox{weakly in }L^2(Q_T).
\end{align*}
The limit function satisfies the boundary condition in the sense
$$
  \frac{u_iq(M)}{p(M)} - \frac{u_i^Dq(M^D)}{p(M^D)}\in L^2(0,T;H_D^1(\Omega)),
$$
with $H^1_D(\Omega) := \{ v \in H^1(\Omega) : v=0$ on $\Gamma^D \}$ 
and it is a weak solution to
\eqref{1.eq}-\eqref{1.bc} in the sense
$$
  \sum_{i=1}^n\bigg(\int_0^T\int_\Omega u_i\pa_t\phi_i dxdt 
	+ \int_\Omega u_i^0(x)\phi(x,0)dx
	\bigg) = \sum_{i=1}^n\int_0^T\int_\Omega p(M)^2\na
	\bigg(\frac{u_iq(M)}{p(M)}\bigg)
	\cdot\na\phi_i\, dxdt, 
$$
for all $\phi_i\in C_0^\infty(\Omega \times [0,T))$. 
\end{thrm}

We also need the assumption $\alpha_i= 1$ for $i=1,\ldots,n$ for the proof of Theorem~\ref{thm.conv}. Indeed, due to the lack of chain rule at the discrete level, 
it is not clear how to identify the weak limit of the term $p(M_\eta)^2 \nabla^\eta(u_{i,\eta} q(M_\eta)/p(M_\eta))$. Another difficulty comes from the degeneracy of $p$ when $M=1$, which prevents the proof of a uniform bound on $\nabla^\eta(u_{i,\eta} q(M_\eta)/p(M_\eta))$ from the entropy inequality~\eqref{2.edi}. Our strategy relies on the uniform upper bound satisfied 
by $M_\eta$ obtained in Theorem~\ref{thm.ex}. Thanks to this bound, the monotonicity of $p$, and the inequality~\eqref{2.edi}, we can establish a uniform 
bound on the $L^2$ norm of $\nabla^\eta(u_{i,\eta} q(M_\eta)/p(M_\eta))$ and identify its weak limit. 
The numerical experiments in Section~\ref{sec.exp.num} seem to indicate that the assumption $\alpha_i=1$ is purely technical and that the scheme still converges in 
the case of different diffusivities.


\section{Existence of finite-volume solutions}\label{sec.ex}

In this section, we prove Theorem \ref{thm.ex}. We proceed by induction.
For $k=0$, we have $u^0\in\overline{\dom}$ with $u^0_i \geq 0$ for $K \in \T$, $i=1,\ldots,n$ by assumption and $M^0 \leq M^* = \sup_{x \in \Omega} \{M^D, M^0(x) \}$ by construction. Assume that there exists a solution
$u_\M^{k-1}$ for some $k\in\{1,\ldots,N_T\}$ such that
$$
  u_K^{k-1}\ge 0, \quad M_K^{k-1} = \sum_{i=1}^n u_{i,K}^{k-1} \leq M^*
	\quad\mbox{for }K\in\T.
$$
The construction of a solution $u_\M^k$ is divided into several steps.

\medskip
\noindent\textbf{Step 1. Definition of a linearized problem.} We introduce the set
\begin{align*}
  Z = \big\{&w_\M=(w_{1,\M},\ldots,w_{n,\M}): w_{i,\sigma}=0 \mbox{ for }
	\sigma\in\E_{\rm ext}^D, \\
	&\|w_{i,\M}\|_{1,2,\M}<\infty\mbox{ for }i=1,\ldots,n\big\}.
\end{align*}
Let $\eps>0$. We define the mapping $F_\eps:Z\to\R^{\theta n}$
by $F_\eps(w_\M)=w_\M^\eps$, with $\theta=\#\T + \#\E^D$, where 
$w_\M^\eps = (w_{1,\M}^\eps,\ldots,w_{n,\M}^\eps)$ is the solution to the 
linear problem
\begin{equation}\label{3.lin}
  \eps \left(-\sum_{\sigma\in\E_K}\tau_\sigma D_{K,\sigma} w_i^\eps + \m(K) w^\eps_{i,K} \right)
  = -\bigg(\frac{\m(K)}{\Delta t}(u_{i,K}-u_{i,K}^{k-1})
	+ \sum_{\sigma\in\E_K}\F_{i,K,\sigma} \bigg),
\end{equation}
for $K\in\T$, $i=1,\ldots,n$ with
\begin{equation}\label{3.bc}
  w_{i,\sigma}^\eps = 0\quad\mbox{for }\sigma\in\E_{\rm ext}^D,\ i=1,\ldots,n.
\end{equation}
Here, $u_{i,K}$ is a function of $w_{i,K}$, defined by
\begin{equation}\label{3.w}
  w_{i,K} = \log\frac{u_{i,K}q(M_K)}{p(M_K)}
	- \log\frac{u_i^Dq(M^D)}{p(M^D)}
	\quad i=1,\ldots,n,
\end{equation}
and $\F_{i,K,\sigma}$ is defined in \eqref{sch2}. Note that $\F_{i,K,\sigma}$ 
depends on $w_\M$ via $u_\M$ and $M$.
It is shown in \cite[Lemma 3.3]{DMZ18} that the mapping $\dom\to\R^n$,
$u_K\mapsto w_K$ is invertible, 
so the function $u_K=u(w_K)$ is well-defined and $u_K\in\dom$ 
(recall definition \eqref{1.O} of $\dom$). The proof in 
\cite[Lemma 3.3]{DMZ18} shows that $M_K\in(0,1)$ such that $\F_{i,K,\sigma}$
is well-defined too. Since $M_K=\sum_{i=1}^n u_{i,K}$, we infer that
$0\le u_{i,K}<1$.
Definitions \eqref{2.vKsigma} and \eqref{2.Dsigma} ensure that
$D_{K,\sigma} w_i^\eps =0$ for all $\sigma\in\E_{{\rm ext},K}^N$.
The existence of a unique solution $w_K^\eps$ to the linear scheme
\eqref{3.lin}-\eqref{3.bc} is now a consequence of \cite[Lemma 9.2]{EGH00}.


\medskip
\noindent\textbf{Step 2. Continuity of $F_\eps$.} We fix $i\in\{1,\ldots,n\}$. We derive first an a priori estimate for $w_{i,\M}^\eps$.
Multiplying \eqref{3.lin} by $w_{i,K}^\eps$,
summing over $K\in\T$ and using the symmetry of $\tau_\sigma$ with respect to 
$\sigma=K|L$, we arrive at
\begin{align}
  \eps\sum_{\sigma\in\E }\tau_\sigma  (D_{\sigma} w_i^\eps)^2+\eps\sum_{K\in\T} \m(K) |w^\eps_{i,K}|^2 
	&= -\sum_{K \in \T}\frac{\m(K)}{\Delta t}(u_{i,K}-u_{i,k}^{k-1}) \, w^{\eps}_{i,K} - \sum_{\substack{\sigma\in\E \\ K=K_\sigma}}\F_{i,K,\sigma}w_{i,K}^\eps 
	  \nonumber \\
	&=: J_1 + J_2, \label{3.aux1}
\end{align}
where in the term $J_2$ the sum is over all edges $\sigma\in\E$, and to each given $\sigma$ 
we associate the cell $K=K_\sigma$.
For the left-hand side, we use the definition~\eqref{2.norm} of the discrete $H^1(\Omega)$ norm
$$
  \eps \sum_{\sigma\in\E }\tau_\sigma  (D_{\sigma} w_i^\eps)^2+\eps\sum_{K\in\T} \m(K) |w^\eps_{i,K}|^2 = \eps\|w_{i,\M}^\eps\|_{1,2,\M}^2.
$$
By the Cauchy-Schwarz inequality and definition \eqref{sch2} of
$\F_{i,K,\sigma}$, we find that
\begin{align*}
  |J_1| &\le \frac{1}{\Delta t}\bigg(\sum_{K\in\T}\m(K)(u_{i,K}-u_{i,K}^{k-1})^2
	\bigg)^{1/2}\bigg(\sum_{K\in\T}\m(K)(w_{i,K}^\eps)^2\bigg)^{1/2} \\
	&\le \frac{1}{\Delta t}\|u_{i,\M}-u_{i,\M}^{k-1}\|_{0,2,\M} \, \|w_{i,\M}^\eps\|_{1,2,\M}, \\
	|J_2| &\le \sum_{\substack{\sigma\in\E \\ K=K_\sigma}}\tau_\sigma
	p_\sigma^2 D_\sigma\bigg(\frac{u_iq(M)}{p(M)}\bigg)D_\sigma w_i^\eps \\
	&\le \bigg(\sum_{\sigma\in\E}\tau_\sigma
	\left(p_\sigma^2\right)^2\bigg( D_\sigma\bigg(\frac{u_iq(M)}{p(M)}\bigg)\bigg)^2\bigg)^{1/2}
	\bigg(\sum_{\sigma\in\E}\tau_\sigma(D_\sigma w_i^\eps)^2\bigg)^{1/2}.
\end{align*}
Since $M_K\in(0,1)$ for all $K\in\T$, $u_{i,K} \, q(M_K)/p(M_K)$ is bounded. Moreover,
$p_\sigma\le p(0)$ as $p$ is decreasing. Hence,
there exists a constant $C(M)>0$ which is independent of $w_{i,\M}^\eps$ such that
$|J_2|\le C(M)\|w_{i,\M}^\eps\|_{1,2,\M}$. This constant does not depend on 
$u_{i,K}\in[0,1)$. Inserting these estimations into \eqref{3.aux1} yields
\begin{equation}\label{3.aux2}
  \sqrt{\eps}\|w^{\eps}_{i,\M}\|_{1,2,\M} \le C(M),
\end{equation}
where $C(M)>0$ is independent of $w_{i,\M}^\eps$.

We turn to the proof of the continuity of $F_\eps$. Let $(w_{\M}^m)_{m\in\N}\in Z$
be such that $w_\M^m\to w_\M$ as $m\to\infty$. Estimate \eqref{3.aux2} shows that
$w_\M^{\eps,m} := F_\eps(w_\M^m)$ is bounded uniformly in $m\in\N$. Thus, there
exists a subsequence of $(w_\M^{\eps,m})$ which is not relabeled such that
$w_\M^{\eps,m}\to w_\M^\eps$ as $m\to\infty$. Passing to the limit $m\to\infty$
in scheme \eqref{3.lin}-\eqref{3.bc} and taking into account the continuity of
the nonlinear functions, we see that $w_{i,\M}^\eps$ is a solution to
\eqref{3.lin}-\eqref{3.bc} and $w_\M^\eps = F_\eps(w_\M)$. Because of the
uniqueness of the limit function, the whole sequence converges, which
proves the continuity.


\medskip
\noindent\textbf{Step 3. Existence of a fixed point.} We claim that the map $F_\eps$ admits a fixed point. We use a topological degree
argument \cite{Dei85}, i.e., we prove that $\delta(I-F_\eps,Z_R,0) = 1$, where
$\delta$ is the Brouwer topological degree and
$$
  Z_R = \{w_\M\in Z: \|w_{i,\M}\|_{1,2,\M} < R\quad\mbox{for }i=1,\ldots,n\}.
$$
Since $\delta$ is invariant by homotopy, it is sufficient to prove that any solution
$(w_\M^\eps,\rho)\in \overline{Z}_R\times[0,1]$ to the fixed-point equation
$w_\M^\eps = \rho F_\eps(w_\M^\eps)$
satisfies $(w_\M^\eps,\rho)\not\in\pa Z_R\times[0,1]$ for sufficiently large values
of $R>0$. Let $(w_\M^\eps,\rho)$ be a fixed point and $\rho\neq 0$, the case $\rho=0$ 
being clear.
Then $w_{i,\M}^\eps$ solves
\begin{equation}\label{3.lin2}
  \eps \left(-\sum_{\sigma\in\E_K}\tau_\sigma D_{K,\sigma} w_i^\eps + \m(K) w^\eps_{i,K} \right)
  = -\rho\bigg(\frac{\m(K)}{\Delta t}(u^\eps_{i,K}-u_{i,K}^{k-1})
	+ \sum_{\sigma\in\E_K}\F^\eps_{i,K,\sigma} \bigg),
\end{equation}
where $\F_{i,K,\sigma}^\eps$ is defined as in \eqref{sch2} with 
$u_\M$ replaced by $u_\M^\eps$ which is related to $w^\eps_\M$ by \eqref{3.w}.

The following discrete entropy inequality is the key argument.

\begin{lem}[Discrete entropy inequality]\label{lem.edi}\sloppy
Let the assumptions of Theorem~\ref{thm.ex} hold. Then for any $\rho\in(0,1]$ and $\eps\in(0,1)$,
\begin{align*}
  & \rho H(u_\M^\eps) + \eps\Delta t\sum_{i=1}^n
	|| w_{i,\M}^\eps ||_{1,2,\M}^2
	+ \rho\Delta t\sum_{i=1}^n I_i(u_\M^\eps) \le \rho H(u_\M^{k-1}), \\
  & \mbox{where }I_i(u_\M^\eps) = \sum_{\sigma\in\E}
	\tau_\sigma (p_\sigma^\eps)^2\bigg(D_\sigma\bigg(
	\sqrt{\frac{u_i^\eps q(M^\eps)}{p(M^\eps)}}\bigg)\bigg)^2, \quad i=1,\ldots,n,
\end{align*}
with obvious notations for $(p_\sigma^\eps)^2$ and $M^\eps$.
\end{lem}

\begin{proof}
We multiply \eqref{3.lin2} by $\Delta t w_{i,K}^\eps$ and sum over $i=1,\ldots,n$
and $K\in\T$. This gives
\begin{align*}
  & \eps \Delta t \sum_{i=1}^n\Bigg(-\sum_{\substack{\sigma\in\E \\ K=K_\sigma}}
	\tau_\sigma w_{i,K}^\eps D_{K,\sigma} w_i^\eps + \sum_{K\in\T} \m(K) |w^\eps_{i,K}|^2 \Bigg)
	+ J_3 + J_4=0, \quad\mbox{where} \\
	& J_3 = \rho\sum_{i=1}^n\sum_{K\in\T}\m(K)(u_{i,K}^\eps-u_{i,K}^{k-1})w_{i,K}^\eps, \\
	& J_4 = \rho\Delta t\sum_{i=1}^n\sum_{\substack{\sigma\in\E \\ K=K_\sigma}}
	\F_{i,K,\sigma}^\eps w_{i,K}^\eps.
\end{align*}
By the symmetry of $\tau_\sigma$ with respect to $\sigma=K|L$, the first term is written as
$$
  \eps \Delta t \sum_{i=1}^n\Bigg(-\sum_{\substack{\sigma\in\E \\ K=K_\sigma}}
	\tau_\sigma w_{i,K}^\eps D_{K,\sigma} w_i^\eps + \sum_{K\in\T} \m(K)
	|w^\eps_{i,K}|^2 \Bigg) 
	= \eps\Delta t\sum_{i=1}^n\| w_{i,\M}^\eps \|_{1,2,\M}^2.
$$
Inserting definition \eqref{3.w} of $w_{i,K}^\eps$ and using the convexity
of $u\mapsto u(\log u-1)+1$, we obtain
\begin{align*}
  J_3 &= \rho\sum_{i=1}^n\sum_{K\in\T}\m(K)(u_{i,K}^\eps-u_{i,K}^{k-1})
	\bigg(\log u_{i,K}^\eps + \log\frac{q(M_K^\eps)}{p(M_K^\eps)}\bigg) \\
	&\phantom{xx}{}- \rho\sum_{i=1}^n\sum_{K\in\T}\m(K)(u_{i,K}^\eps-u_{i,K}^{k-1})
	\bigg(\log u_{i}^D + \log\frac{q(M^D)}{p(M^D)}\bigg) \\
  &\ge \rho\sum_{K\in\T}\m(K)\big(h(u_K^\eps)-h(u_K^{k-1})\big)
	- \rho\sum_{i=1}^n\m(K)(u_{i,K}^\eps-u_{i,K}^{k-1})\frac{\pa h}{\pa u_i}(u^D) \\
	&= \rho\sum_{K\in\T}\m(K)\big(h(u_K^\eps) - (u_K^\eps-u^D)\cdot h'(u^D)\big) \\
	&\phantom{xx}{}
	- \rho\sum_{K\in\T}\m(K)\big(h(u_K^{k-1}) - (u_K^{k-1}-u^D)\cdot h'(u^D)\big) \\
	&= \rho\sum_{K\in\T}\m(K)\big(h^*(u_K^\eps|u^D) - h^*(u_K^{k-1}|u^D)\big)
	= \rho\big(H(u_\M^\eps)-H(u_\M^{k-1})\big).
\end{align*}
We abbreviate $v_{i,K}^\eps := u_{i,K}^\eps q(M_K^\eps)/p(M_K^\eps)$. Then
\begin{align*}
  J_4 &= -\rho\Delta t\sum_{i=1}^n\sum_{\substack{\sigma\in\E \\ K=K_\sigma}}
	\F_{i,K,\sigma}^\eps D_{K,\sigma}(w_i^\eps) \\
	&= \rho\Delta t\sum_{i=1}^n \sum_{\substack{\sigma\in\E \\ K=K_\sigma}}
	\tau_\sigma (p_\sigma^\eps)^2(v^\eps_{i,K,\sigma}-v_{i,K}^\eps)
	(\log v_{i,K,\sigma}^\eps - \log v_{i,K}^\eps).
\end{align*}
The elementary inequality $(x-y)(\log x-\log y)\ge 4(\sqrt{x}-\sqrt{y})^2$ 
for any $x$, $y>0$ implies that
$$
  J_4 \ge 4\rho\Delta t\sum_{i=1}^n \sum_{\sigma\in\E }
	\tau_\sigma (p_\sigma^\eps)^2\bigg(D_\sigma\bigg(\sqrt{
	\frac{u_i^\eps q(M^\eps)}{p(M^\eps)}}\bigg)\bigg)^2.
$$
Putting all the estimations together completes the proof.
\end{proof}

We proceed with the topological degree argument. The previous lemma implies that
$$
  \eps\Delta t\sum_{i=1}^n
	|| w_{i,\M}^\eps ||_{1,2,\M}^2
	\le \rho H(u_\M^{k-1}) \le H(u_\M^{k-1}).
$$
Then, if we define
\begin{equation*}
  R := \left(\frac{H(u_\M^{k-1})}{\eps\Delta t}\right)^{1/2}+1,
\end{equation*}
we conclude that $w_\M^\eps\not\in\pa Z_R$ and $\delta(I-F_\eps,Z_R,0)=1$.
Thus, $F_\eps$ admits a fixed point.


\medskip
\noindent\textbf{Step 4. Limit $\eps\to 0$.} We recall that $u_\M^\eps\in\overline\dom$. Thus, up to a subsequence,
$u_\M^\eps\to u_\M\in\overline\dom$ as $\eps\to 0$. We deduce from
\eqref{3.aux2} that there exists a subsequence (not relabeled) such that
$\eps w_{i,K}^\eps\to 0$ for any $K\in\T$ and $i=1,\ldots,n$. In order to
pass to the limit in the fluxes $\F_{i,K,\sigma}^\eps$, we need to show that
$M_K=\sum_{i=1}^n u_{i,K}<1$ for any $K\in\T$. To this end, we establish the following result:

\begin{lem}[$L^{2}$ estimate]
\label{lem.borne.sup}
Let the assumptions of Theorem \ref{thm.ex} hold. Then for all $\eps > 0$, 
there exists a constant $C > 0$ depending on $H(u^{k-1}_\M)$, $\Omega$, $\Delta t$, the mesh 
$\T$, and $M^* = \sup_{x \in \Omega} \{ M^D, M^0(x) \}$ such that
\begin{align}
\label{in.borne.sup}
\sum_{K \in \T} \m(K) \left([M^\eps_K -M^*]^+\right)^2 \leq C \, \sqrt{\eps},
\end{align}
where $z^+ = \max\{z,0\}$.
\end{lem}

\begin{proof}
Let $\eps > 0$ be fixed. Then, summing \eqref{3.lin2} over $i$, we obtain
\begin{align*}
\eps \sum_{i=1}^n \left(- \sum_{\sigma \in \E_K} \tau_\sigma D_{K,\sigma} w^\eps_i + \m(K) w^\eps_{i,K} \right) &+
\m(K) \frac{M^{\eps}_{K}-M^{k-1}_{K}}{\Delta t} \\
& +\sum_{i=1}^n \sum_{\sigma \in \E_K} \F^{\eps}_{i,K,\sigma} = 0 \quad \mbox{for all }K \in \T.
\end{align*}
Multiplying this equation by $\Delta t [M^\eps_K-M^*]^+$, summing over $K \in \T$, 
and using $\frac{1}{2}(x^2-y^2) \leq x(x-y)$, we obtain
\begin{align*}
\sum_{K\in\T}\frac{\m(K)}{2} \left( [M^\eps_K-M^*]^2_+ - [M^{k-1}_K-M^*]^2_+ \right) \leq J_5+J_6+J_7,
\end{align*}
where
\begin{align*}
J_5 &= -\Delta t\sum_{i=1}^n \sum_{\substack{\sigma\in\E \\ K=K_\sigma}} \F^{\eps}_{i,K,\sigma} [M^\eps_K-M^*]_+,\\
J_6 &= \eps \Delta t \sum_{i=1}^n \sum_{\substack{\sigma\in\E \\ K=K_\sigma}} \tau_\sigma D_{K,\sigma} w^\eps_i [M^\eps_K - M^*]^+,\\
J_7 &= -\eps \Delta t \sum_{i=1}^n \sum_{K \in \T} \m(K) w^\eps_{i,K} [M^\eps_K - M^*]^+.
\end{align*}
We use discrete integration by parts to rewrite $J_5$ as
\begin{align*}
J_5 = - \Delta t \sum_{\substack{\sigma\in\E \\ K=K_\sigma}} \tau_{\sigma} (p^\eps_\sigma)^2 D_{K,\sigma} \left(\frac{M^\eps q(M^\eps)}{p(M^\eps)}\right) \, D_{K,\sigma} [M^\eps - M^*]^+.
\end{align*}
We assume that for $\sigma \in \E$ we have $M^\eps_{K,\sigma} \geq M^\eps_K$. Then, since the function $M \mapsto M q(M)/p(M)$ is increasing (see definition 
\eqref{1.q}), we deduce that $D_{K,\sigma}(M^\eps \, q(M^\eps)/p(M^\eps)) \geq 0$.
We distinguish the following cases:
\begin{itemize}
	\item  $M^* \geq M^\eps_{K,\sigma} \geq M^\eps_K$ $\Rightarrow$
	$D_{K,\sigma}[M^\eps-M^*]^+=0$;
	\item $M^\eps_{K,\sigma} \geq M^* \geq M^\eps_K$ $\Rightarrow$
	$D_{K,\sigma}[M^\eps-M^*]^+=M^\eps_{K,\sigma}-M^* \geq 0$;
	\item $M^\eps_{K,\sigma} \geq M^\eps_K \geq M^*$ $\Rightarrow$
	$D_{K,\sigma}[M^\eps-M^*]^+=M^\eps_{K,\sigma}-M^\eps_K \geq 0$.
\end{itemize}
This implies that $D_{K,\sigma}(M^\eps  q(M^\eps)/p(M^\eps)) D_{K,\sigma}[M^\eps-M^*]^+ \geq 0$ if $M^\eps_{K,\sigma} \geq M^\eps_K$. A similar argument shows that $D_{K,\sigma}(M^\eps  q(M^\eps)/p(M^\eps)) D_{K,\sigma}[M^\eps-M^*]^+ \geq 0$ also in the case $M^\eps_K \geq M^\eps_{K,\sigma}$ and we deduce that $J_5 \leq 0$.

For $J_6$, we apply discrete integration by parts and the Cauchy-Schwarz inequality:
$$
|J_6| \leq \eps^{1/2} \left(\eps \Delta t \sum_{i=1}^n \sum_{\sigma \in \E} \tau_\sigma (D_\sigma w^\eps_i)^2 \right)^{1/2} \left(\Delta t \sum_{\sigma \in \E} \tau_\sigma (D_{\sigma} [M^\eps-M^*]^+)^2 \right)^{1/2}.
$$
It follows from Lemma~\ref{lem.edi} and the $L^\infty$ bound $M^\eps_K \leq 1$ 
for $K \in \T$ that
$$
|J_6| \leq 2 H(u^{k-1}_\M)^{1/2} \, (1+M^*) \left(\Delta t\sum_{\sigma \in \E} \tau_\sigma \right)^{1/2} \eps^{1/2}.
$$
Finally, we use the Cauchy-Schwarz inequality together with Lemma~\ref{lem.edi} 
and then the $L^\infty$ bound $M^\eps_K \leq 1$ for $K \in \T$ to estimate $J_7$:
\begin{align*}
|J_7| &\leq \eps^{1/2} H(u^{k-1}_\M)^{1/2} \left(\Delta t\sum_{K\in\T} \m(K) \left([M^\eps_K - M^*]^+\right)^2\right)^{1/2} \\
&\leq H(u^{k-1}_\M)^{1/2} \, (1+M^*) \Delta t^{1/2} \,\m(\Omega)^{1/2} \eps^{1/2}.
\end{align*}
Gathering all the previous estimates, we deduce the existence of a constant $C > 0$ 
such that~\eqref{in.borne.sup} holds.
\end{proof}

We conclude from Lemma~\ref{lem.borne.sup} that passing to the limit $\eps\to 0$ in~\eqref{in.borne.sup} that
\begin{align*}
\sum_{K\in\T} \m(K) \left([M_K-M^*]^+\right)^2 \leq 0,
\end{align*}
recall that $M^\eps_K \to M_K$ as $\eps \to 0$ for $K \in \T$. This shows that $M_K \leq M^* <1$ for all $K\in\T$.
We can perform the limit $\eps\to 0$ in~\eqref{3.lin2}, which completes the proof of Theorem \ref{thm.ex}.


\section{A priori estimates}\label{sec.apriori}

In this section, we establish some uniform estimates for the solutions to scheme \eqref{sch.ic}-\eqref{2.psigma}. 

\subsection{Gradient estimate}

We deduce the following gradient estimate from the entropy inequality \eqref{2.edi}.

\begin{lem}[Gradient estimate]\label{lem.grad}
Let the assumptions of Theorem \ref{thm.ex} hold. Then there exists a constant
$C_1>0$ only depending on $H(u_\M^0)$, $\Omega$, $q$, $p$, and the upper bound $M^*$ defined in Theorem~\ref{thm.ex} such that
$$
  \sum_{k=1}^{N_T}\Delta t\bigg\|\frac{u_{i,\M}^kq(M_\M^k)}{p(M_\M^k)}
	\bigg\|^2_{1,2,\M} \le C_1 \quad \mbox{for all }1 \leq i \leq n.
$$
\end{lem}

\begin{proof}
Let $i \in \{1,\ldots,n \}$. Thanks to the uniform $L^\infty$ bound for $u_\M^k$, 
it is sufficient to show that there exists a constant $C>0$ independent of
$\Delta x$ and $\Delta t$ such that
$$
   \sum_{k=1}^{N_T}\Delta t\bigg|\frac{u_{i,\M}^kq(M_\M^k)}{p(M_\M^k)}
	\bigg|^2_{1,2,\M} \le C.
$$
To prove this estimate, we start from the following bound which comes from
the discrete entropy inequality~\eqref{2.edi}:
\begin{equation}\label{4.aux}
  \sum_{k=1}^{N_T}\Delta t\sum_{\sigma\in\E }\tau_\sigma
	 \bigg(D_\sigma\bigg(\sqrt{\frac{u_i^k q(M^k)}{p(M^k)}}\bigg)\bigg)^2
	\le \frac{H(u_\M^0)}{p(M^*)^2}.
\end{equation}
Using the inequality $x^2-y^2\le 2x(x-y)$ and $u_{i,K,\sigma}^k\le 1$, 
we can write
\begin{align*}
  \sum_{k=1}^{N_T}\Delta t\sum_{\sigma\in\E}
	\tau_\sigma\bigg(D_\sigma\bigg(\frac{u_i^k q(M^k)}{p(M^k)}\bigg)\bigg)^2 
	&\le 4\sum_{k=1}^{N_T}\Delta t\sum_{\sigma\in\E }
	\tau_\sigma\frac{u_{i,K,\sigma}^k q(M_{K,\sigma}^k)}{p(M_{K,\sigma}^k)}
	\bigg(D_\sigma\bigg(\sqrt{\frac{u_i^k q(M^k)}{p(M^k)}}\bigg)\bigg)^2 \\
  &\le 4\sum_{k=1}^{N_T}\Delta t\sum_{\sigma\in\E}
	\tau_\sigma\frac{q(M_{K,\sigma}^k)}{p(M_{K,\sigma}^k)}
	\bigg(D_\sigma\bigg(\sqrt{\frac{u_i^k q(M^k)}{p(M^k)}}\bigg)\bigg)^2.
\end{align*}
Thanks to~\cite[Lemma 3.4]{DMZ18}, we know that the function $x \mapsto \sqrt{q(x)/p(x)}$ is strictly increasing for $x \in (0,1)$. We use the $L^\infty$ bound $M^k_K \leq M^*$ for $K \in \T$ given in Theorem~\ref{thm.ex} to conclude that
$$
  \sum_{k=1}^{N_T}\Delta t\sum_{\sigma\in\E}
	\tau_\sigma\bigg(D_\sigma\bigg(\frac{u_i^k q(M^k)}{p(M^k)}\bigg)\bigg)^2  \le \frac{4q(M^*)}{p(M^*)}\sum_{k=1}^{N_T}\Delta t
	\sum_{\sigma\in\E }\tau_\sigma
	\bigg(D_\sigma\bigg(\sqrt{\frac{u_i^k q(M^k)}{p(M^k)}}\bigg)\bigg)^2.
$$
In view of \eqref{4.aux}, this shows the lemma.
\end{proof}


\subsection{Estimate for the time difference}

We wish to apply the compactness result from \cite{ACM17}. To this end, we need
to prove a uniform estimate on the difference $u_{i,K}^k-u_{i,K}^{k-1}$.

\begin{lem}[Time estimate]\label{lem.time}
Let the assumptions of Theorem \ref{thm.ex} hold. Then there exists a constant
$C_2>0$ not depending on $\Delta x$ and $\Delta t$ such that for all
$i \in \{ 1,\ldots,n \}$ and $\phi\in C_0^\infty(Q_T)$,
$$
  \sum_{k=1}^{N_T}\Delta t\sum_{K\in\T}(u_{i,K}^k-u_{i,K}^{k-1})
	\phi(x_K,t_k) \le C_2\Delta t\|\na\phi\|_{L^\infty(Q_T)}.
$$
\end{lem}

\begin{proof}
We abbreviate $\phi_K^k:=\phi(x_K,t_k)$ and fix $i\in\{1,\ldots,n\}$. We multiply \eqref{sch1} by $\Delta t\phi_K^k$ and sum over $K\in\T$ and 
$k=1,\ldots,N_T$
\begin{align*}
  \sum_{k=1}^{N_T}\sum_{K\in\T}\m(K)&(u_{i,K}^k-u_{i,K}^{k-1})\phi_K^k
	= -\sum_{k=1}^{N_T}\Delta t\sum_{\substack{\sigma\in\E \\ K=K_\sigma}}
	\F_{i,K,\sigma}^k\phi_K^k 
	=: J_8.
\end{align*}
Inserting the definition of $\F_{i,K,\sigma}^k$ and using the symmetry of
$\tau_\sigma$ with respect to $\sigma=K|L$, we find that
$$
  J_8 = -\sum_{k=1}^{N_T}\Delta t\sum_{\substack{\sigma\in\E \\ K=K_\sigma}}
	\tau_\sigma(p_\sigma^k)^2 D_{K,\sigma}\bigg(\frac{u_i^kq(M^k)}{p(M^k)}\bigg)
	D_{K,\sigma} \phi^k.
$$
Using the Cauchy-Schwarz inequality, we obtain $|J_8| \leq J_{80} J_{81}$, where
\begin{align*}
  J_{80} &= \bigg(\sum_{k=1}^{N_T}\Delta t|\phi_\M^k|_{1,2,\M}^2\bigg)^{1/2}, \\
  J_{81} &= \bigg(\sum_{k=1}^{N_T}\Delta t\sum_{\sigma\in\E}
	\tau_\sigma \big((p_\sigma^k)^2\big)^2 \bigg[D_\sigma\bigg(\frac{u_i^kq(M^k)}{p(M^k)}\bigg)\bigg]^2\bigg)^{1/2}.
\end{align*}
It follows from the mesh properties \eqref{2.regmesh} and \eqref{2.estmesh} that
\begin{align*}
  J_{80} &\le \|\na\phi\|_{L^\infty(Q_T)}\bigg(\sum_{k=1}^{N_T}\Delta t
	\sum_{\sigma\in\E}\m(\sigma)\dist_\sigma\bigg)^{1/2} \\
	&\le \frac{1}{\xi^{1/2}} \|\na\phi\|_{L^\infty(Q_T)}\bigg(\sum_{k=1}^{N_T}\Delta t
	\sum_{K\in\T}\sum_{\sigma\in\E_K}\m(\sigma)\dist(x_K,\sigma)\bigg)^{1/2} \\
  &\le \frac{2^{1/2}}{\xi^{1/2}}\|\na\phi\|_{L^\infty(Q_T)}\bigg(\sum_{k=1}^{N_T}
	\Delta t\sum_{K\in\T}\m(K)\bigg)^{1/2}
	= \sqrt{\frac{2\m(\Omega)T}{\xi}}\|\na\phi\|_{L^\infty(Q_T)}.
\end{align*}
By Lemma \ref{lem.grad}, $J_{81}\le C_1p(0)^2$. This shows that $|J_8|\le
C_2\Delta t\|\na\phi\|_{L^\infty(Q_T)}$, concluding the proof.
\end{proof}


\subsection{Lower bound for the entropy production term}

In this section we establish a discrete counterpart of inequality~\eqref{1.ineq}.

\begin{lem}[Lower bound for the entropy production]\label{lem.lower}
Let the assumptions of Theorem \ref{thm.ex} hold. Then there exists a constant
$C_3>0$ depending on $p$, $q$, $a$, $b$, and $\kappa$ such that for $k=1,\ldots,N_T$,
\begin{align}\label{3.I}
  \sum_{i=1}^n I_i(u_\M^k) \ge \frac12 \sum_{i=1}^n
	\sum_{\sigma\in\E }\tau_\sigma \beta^k_{K,\sigma}
	\big(D_\sigma\sqrt{u_i^k}\big)^2
	+ C_3\sum_{\sigma\in\E }\tau_\sigma
	\frac{(M_\sigma^k)^{a-1}(D_\sigma M^k)^2}{(1-M_\sigma^k)^{1+b+\kappa}},
\end{align}
where $M_\sigma^k = \theta_\sigma M_K^k + (1-\theta_\sigma)M_{K,\sigma}^k$
for some $\theta_\sigma\in(0,1)$, 
$$
  \beta^k_{K,\sigma} = \min\big\{p(M_K^k)q(M_K^k),p(M_{K,\sigma}^k)
	q(M_{K,\sigma}^k)\big\},
$$
and we recall that $I_i(u_\M^k)$ is defined
in \eqref{2.ed}.
\end{lem}

\begin{proof}
To simplify the presentation, we omit the superindex $k$ throughout the proof.
Summing definition \eqref{2.ed} for $I_i(u_\M)$ over $i=1,\ldots,n$ 
and setting $f(x)=\sqrt{q(x)/p(x)}$, we obtain
$$
  I := \sum_{i=1}^n I_i(u_\M)
	= \sum_{i=1}^n\sum_{\substack{\sigma\in\E \\ K=K_\sigma}}\tau_\sigma
	p_\sigma^2\big(D_{K,\sigma}(\sqrt{u_i} f(M))\big)^2. 
$$
We split the sum into two 
parts and use the product rule for finite volumes. Then $I=J_{90}+J_{91}$, where
\begin{align*}
  J_{90} &= \sum_{i=1}^n\sum_{\substack{\sigma\in\E \\ K=K_\sigma}}\tau_\sigma
	p_\sigma^2\big(\sqrt{u_{i,K,\sigma}}D_{K,\sigma}(f(M)) + D_{K,\sigma}(\sqrt{u_i})
	f(M_K)\big)^2\mathbf{1}_{\{M_{K,\sigma}\ge M_K\}}, \\
  J_{91} &= \sum_{i=1}^n\sum_{\substack{\sigma\in\E \\ K=K_\sigma}}\tau_\sigma
	p_\sigma^2\big(\sqrt{u_{i,K}}D_{K,\sigma}(f(M)) + D_{K,\sigma}(\sqrt{u_i})
	f(M_{K,\sigma})\big)^2\mathbf{1}_{\{M_{K,\sigma} < M_K\}}.
\end{align*}
A Taylor expansion of $f$ around $M_{K,\sigma}$ gives
\begin{align*}
  J_{90} &= \sum_{i=1}^n\sum_{\substack{\sigma\in\E \\ K=K_\sigma}}\tau_\sigma
	p_\sigma^2\big(\sqrt{u_{i,K,\sigma}}D_{K,\sigma}(M) f'(M_\sigma) 
	+ D_{K,\sigma}(\sqrt{u_i})f(M_K)\big)^2\mathbf{1}_{\{M_{K,\sigma}\ge M_K\}}, \\
  J_{91} &= \sum_{i=1}^n\sum_{\substack{\sigma\in\E \\ K=K_\sigma}}\tau_\sigma
	p_\sigma^2\big(\sqrt{u_{i,K}}D_{K,\sigma}(M) f'(M_\sigma) 
	+ D_{K,\sigma}(\sqrt{u_i})f(M_{K,\sigma})\big)^2\mathbf{1}_{\{M_{K,\sigma} < M_K\}},
\end{align*}
where $M_\sigma = \theta_\sigma M_{K,\sigma} + (1-\theta_\sigma)M_K$ for some
$\theta_\sigma\in(0,1)$ and for $K\in\T$ and $\sigma\in\E_K$. 

We consider the term $J_{90}$ first. Expanding the square gives three terms,
$J_{90} = J_{901} + J_{902} + J_{903}$, where
\begin{align*}
  J_{901} &= \sum_{i=1}^n\sum_{\sigma\in\E }\tau_\sigma
	p_\sigma^2 f(M_K)^2\big(D_{\sigma}(\sqrt{u_i})\big)^2
	\mathbf{1}_{\{M_{K,\sigma}\ge M_K\}}, \\
  J_{902} &= 2\sum_{i=1}^n\sum_{\substack{\sigma\in\E \\ K=K_\sigma}}\tau_\sigma
	p_\sigma^2\sqrt{u_{i,K,\sigma}}D_{K,\sigma}(\sqrt{u_i})f'(M_\sigma)f(M_K)
	D_{K,\sigma}(M)\mathbf{1}_{\{M_{K,\sigma}\ge M_K\}}, \\
	J_{903} &= \sum_{i=1}^n\sum_{\sigma\in\E }\tau_\sigma
	p_\sigma^2 u_{i,K,\sigma} f'(M_\sigma)^2(D_{\sigma}M)^2
	\mathbf{1}_{\{M_{K,\sigma}\ge M_K\}} \\
	&= \sum_{\sigma\in\E }\tau_\sigma p_\sigma^2
	M_{K,\sigma}f'(M_\sigma)^2(D_{\sigma}M)^2
	\mathbf{1}_{\{M_{K,\sigma}\ge M_K\}},
\end{align*}
and in the last equality we used the identity 
$\sum_{i=1}^n u_{i,K,\sigma}=M_{K,\sigma}$.

Definition~\eqref{2.psigma} of $p_{\sigma}^2$ implies that $p_\sigma^2 \geq p(M_K)^2/2$. Then, by definition of $f$,
$$
  J_{901} = \frac{1}{2}\sum_{i=1}^n\sum_{\sigma\in\E }\tau_\sigma
	p(M_K) q(M_K)\big(D_{\sigma}(\sqrt{u_i})\big)^2
	\mathbf{1}_{\{M_{K,\sigma}\ge M_K\}}.
$$
The function $f$ is strictly increasing \cite[Lemma 3.4]{DMZ18}. Since
$x(x-y)\ge\frac12(x^2-y^2)$, it follows that
\begin{align*}
  J_{902} &\ge \sum_{i=1}^n\sum_{\substack{\sigma\in\E \\ K=K_\sigma}}\tau_\sigma
  p_\sigma^2(u_{i,K,\sigma}-u_{i,K})f'(M_\sigma) f(M_K) D_{K,\sigma}(M)
	\mathbf{1}_{\{M_{K,\sigma}\ge M_K\}} \\
	&= \sum_{\sigma\in\E }\tau_\sigma p_\sigma^2
	(D_{\sigma}M)^2f'(M_\sigma)f(M_K)\mathbf{1}_{\{M_{K,\sigma}\ge M_K\}} \ge 0.
\end{align*}

It remains to estimate $J_{903}$. For this, we set 
$J_{903} = \sum_{\sigma\in\E}J_{903}(\sigma)$, where
$$
  J_{903}(\sigma) = \tau_\sigma p_\sigma^2 M_{K,\sigma}f'(M_\sigma)^2
	(D_{\sigma}M)^2\mathbf{1}_{\{M_{K,\sigma}\ge M_K\}}.
$$
Thanks to \cite[Lemma 3.1]{DMZ18}, there exists a constant $C_{pq}$ such that
$$
  \lim_{M\to 1}\frac{p(M)q(M)}{(1-M)^{1-b+\kappa}} = C_{pq}\in(0,\infty).
$$
We deduce that there exists $\delta\in(0,1/2)$ such that
for all $M_\sigma>1-\delta$,
\begin{equation}\label{4.pq}
  \frac{p(M_\sigma)q(M_\sigma)}{(1-M_\sigma)^{1-b+\kappa}} \ge \frac{C_{pq}}{2}.
\end{equation}
We distinguish the cases (i) $0\le M_\sigma\le 1-\delta$ and
(ii) $1-\delta<M_\sigma<1$.

Consider first case (i). Modifying slightly the proof of \cite[Lemma 3.4]{DMZ18},
it holds that for all $0\le M_\sigma\le 1-\delta$,
$$
  f'(M_\sigma) \ge \frac{a}{2M_\sigma}f(M_\sigma), \quad
	p(M_\sigma)q(M_\sigma) \ge \frac{p(1-\delta)^2M_\sigma^a}{p(0)^2(a+1)}.
$$
On the set $\{M_{K,\sigma} \geq M_K \}$ we have $M_{K,\sigma}\ge M_\sigma\ge M_K$, and thus, $p_\sigma^2 \geq p(M_K)^2/2 \geq p(M_{\sigma})^2/2$. Therefore,
taking into account the definition of $f$,
\begin{align*}
  J_{903}(\sigma) &\ge \tau_\sigma \frac{p(M_\sigma)^2}{2} M_{K,\sigma}\frac{a^2}{4M_\sigma^2}
	f(M_\sigma)^2(D_\sigma M)^2\mathbf{1}_{\{M_{K,\sigma}\ge M_K\}} \\
  &= \frac{a^2}{8}
	\tau_\sigma p(M_\sigma)q(M_\sigma)\frac{M_{K,\sigma}}{M_\sigma^2}(D_\sigma M)^2
	\mathbf{1}_{\{M_{K,\sigma}\ge M_K\}} \\
	&\ge \frac{a^2 p(1-\delta)^2}{8(a+1)p(0)^2}
	\tau_\sigma M_\sigma^{a-1}\frac{M_{K,\sigma}}{M_\sigma}(D_\sigma M)^2
	\mathbf{1}_{\{M_{K,\sigma}\ge M_K\}} \\
	&\ge \frac{a^2 p(1-\delta)^2}{8(a+1)p(0)^2}
	\tau_\sigma M_\sigma^{a-1}(D_\sigma M)^2\mathbf{1}_{\{M_{K,\sigma}\ge M_K\}},
\end{align*}
where we used $M_{K,\sigma}\ge M_\sigma$ in the last inequality.
Since $M_\sigma\le 1-\delta$, we have $(1-M_\sigma)^{1+b+\kappa}
\ge \delta^{1+b+\kappa}$ and consequently,
$$
  J_{903}(\sigma) \ge \frac{a^2 p(1-\delta)^2\delta^{1+b+\kappa}}{8(a+1)p(0)^2}
	\frac{\tau_\sigma M_\sigma^{a-1}}{(1-M_\sigma)^{1+b+\kappa}}
	(D_\sigma M)^2\mathbf{1}_{\{M_{K,\sigma}\ge M_K\}}.
$$

In case (ii), using $M_{K,\sigma}\ge M_\sigma>1-\delta$ and
$p_\sigma^2 \ge p(M_K)^2/2 \ge p(M_\sigma)^2/2$, we find that
\begin{align*}
  J_{903}(\sigma) &\ge \frac{1}{2} (1-\delta)\tau_\sigma p(M_\sigma)^2 f'(M_\sigma)^2
  (D_\sigma M)^2\mathbf{1}_{\{M_{K,\sigma}\ge M_K\}} \\
	&\ge \frac{1}{2} (1-\delta)\tau_\sigma p(M_\sigma)q(M_\sigma)
	\bigg(\frac{f'(M_\sigma)}{f(M_\sigma)}\bigg)^2
	(D_\sigma M)^2\mathbf{1}_{\{M_{K,\sigma}\ge M_K\}}.
\end{align*}
The proof of \cite[Lemma 3.4]{DMZ18} shows that there exists a constant $C_4>0$
such that 
$$
  \frac{f'(x)}{f(x)} \ge \frac{C_4}{(1-x)^{1+\kappa}}\quad\mbox{for }\frac12<x<1.
$$
Hence, together with \eqref{4.pq}, we infer that
\begin{align*}
  J_{903}(\sigma) &\ge \frac{1}{2}(1-\delta)C_4^2\tau_\sigma
	\frac{p(M_\sigma)q(M_\sigma)}{(1-M_\sigma)^{1-b+\kappa}}
	(1-M_\sigma)^{-1-b-\kappa}(D_\sigma M)^2\mathbf{1}_{\{M_{K,\sigma}\ge M_K\}} \\
  &\ge \frac14(1-\delta)C_{pq}C_4^2\tau_\sigma(1-M_\sigma)^{-1-b-\kappa}
	(D_\sigma M)^2\mathbf{1}_{\{M_{K,\sigma}\ge M_K\}} \\
	&\ge \frac14(1-\delta)C_{pq}C_4^2\tau_\sigma
	\frac{M_\sigma^{a-1}}{(1-M_\sigma)^{1+b+\kappa}}
	(D_\sigma M)^2\mathbf{1}_{\{M_{K,\sigma}\ge M_K\}},
\end{align*}
where in the last step we used $M_\sigma\le 1$ and $a\geq 1$. 
We have proved that in
both cases (i) and (ii), there exists a constant $C_5>0$ such that
\begin{align*}
  J_{903} \ge C_5\sum_{\sigma\in\E }\tau_\sigma
	\frac{M_\sigma^{a-1}}{(1-M_\sigma)^{1+b+\kappa}}
	(D_\sigma M)^2\mathbf{1}_{\{M_{K,\sigma}\ge M_K\}}.
\end{align*}

Similarly, we expand the square in $J_{91}$ such that 
$J_{91}=J_{911}+J_{912}+J_{913}$, where
\begin{align*}
  J_{911} &= \sum_{i=1}^n\sum_{\sigma\in\E }
	\tau_\sigma p_\sigma^2f(M_{K,\sigma})^2(D_{\sigma}(\sqrt{u_i}))^2
	\mathbf{1}_{\{M_{K,\sigma} < M_K\}}, \\
	J_{912} &= 2\sum_{i=1}^n\sum_{\substack{\sigma\in\E \\ K=K_\sigma}}
	\tau_\sigma p_\sigma^2 \sqrt{u_{i,K}}D_{K,\sigma}(\sqrt{u_i})f'(M_\sigma)f(M_{K,\sigma})
	D_{K,\sigma}(M)\mathbf{1}_{\{M_{K,\sigma} < M_K\}}, \\
  J_{913} &= \sum_{i=1}^n\sum_{\sigma\in\E }
	\tau_\sigma p_\sigma^2 u_{i,K}f'(M_{\sigma})^2(D_{\sigma}M)^2
	\mathbf{1}_{\{M_{K,\sigma} < M_K\}}.
\end{align*}
Arguing as for the expressions $J_{901}$ and $J_{902}$, we obtain $J_{912}\ge 0$ and
$$
  J_{911} = \frac12 \sum_{i=1}^n\sum_{\sigma\in\E}
	\tau_\sigma p(M_{K,\sigma})q(M_{K,\sigma})(D_{\sigma}(\sqrt{u_i}))^2
	\mathbf{1}_{\{M_{K,\sigma} < M_K\}}.
$$
The terms in $J_{913}$ are studied as before for the cases
$0\le M_\sigma\le 1-\delta$ and $M_\sigma>1-\delta$. Similar computations
lead to the existence of a constant $C_6>0$ such that
\begin{align*}
  J_{913} &\ge C_6\sum_{\sigma\in\E }\tau_\sigma
	\frac{M_\sigma^{a-1}}{(1-M_\sigma)^{1+b+\kappa}}
	(D_\sigma M)^2\mathbf{1}_{\{M_{K,\sigma} < M_K\}}.
\end{align*}
We put together the estimates for $J_{901}$ and $J_{911}$,
\begin{equation}\label{J91J101}
  J_{901}+J_{911} \ge \frac12 \sum_{\sigma\in\E }
	\tau_\sigma \min\big\{p(M_K)q(M_K), p(M_{K,\sigma})q(M_{K,\sigma})\big\}
	(D_\sigma\sqrt{u_i})^2.
\end{equation}
and add $J_{903}$ and $J_{913}$,  
\begin{align}\label{J93J103}
  J_{903}+J_{913} &\ge \min\{C_5,C_6\}
	\sum_{\sigma\in\E }\tau_\sigma
	\frac{M_\sigma^{a-1}}{(1-M_{\sigma})^{1+b+\kappa}}
	(D_\sigma M)^2.
\end{align}
Note that $J_{902}+J_{912}\ge 0$. Then
$I \ge (J_{901} + J_{911}) + (J_{903}+J_{913})$ and inserting estimates
\eqref{J91J101} and \eqref{J93J103}, we finish the proof.
\end{proof}


\section{Convergence of solutions}\label{sec.convsol}

We wish to prove Theorem \ref{thm.conv}. Before proving the convergence of the
scheme, we show some compactness properties for the solutions of scheme \eqref{sch.ic}-\eqref{2.psigma}.

\subsection{Compactness properties}

Applying Theorem 3.9 in \cite{ACM17}, we obtain the following result.

\begin{prpstn}[Almost everywhere convergence]\label{prop.conv}
Let the assumptions of Theorem \ref{thm.conv} hold and
let $(u_\eta)_{\eta>0}$ be a family of discrete solutions
to scheme \eqref{sch.ic}-\eqref{2.psigma} constructed in Theorem \ref{thm.ex}.
Then there exists a subsequence of $(u_\eta)_{\eta>0}$, which is not relabeled,
and a function $u=(u_1,\ldots,u_n)\in L^\infty(Q_T)^n$ such that, as $\eta\to 0$,
$$
  u_{i,\eta}\to u_i\ge 0\quad\mbox{a.e. in }Q_T,\ i=1,\ldots,n.
$$
Moreover, there exists $M\in L^\infty(Q_T)$ such that
$$
  M_\eta = \sum_{i=1}^n u_{i,\eta} \to M = \sum_{i=1}^n u_i < 1
	\quad\mbox{a.e. in }Q_T.
$$
\end{prpstn}

\begin{proof}
Assumptions (A$_{\rm x}$1) and (A$_{\rm x}$3) in \cite[Theorem 3.9]{ACM17}
are satisfied due to the choice of finite volumes. Assumption (A$_{\rm t}$)
is always fulfilled for one-step methods like the implicit Euler discretization.
Assumptions (a) and (b) are a consequence of the $L^\infty$ bound, 
while Lemma \ref{lem.time} ensures assumption (c). Thus, the result follows
directly from \cite[Theorem 3.9]{ACM17}.
\end{proof}

The gradient estimate in Lemma \ref{lem.grad} shows that the discrete gradient
of $u_{i,\eta} q(M_\eta)/p(M_\eta)$ converges weakly in $L^2(Q_T)$ (up to a subsequence)
to some function. The following lemma shows that the limit can be identified with
$\na(u_iq(M)/p(M))$.

\begin{lem}[Convergence of the gradient]\label{lem.convgrad}
Let the assumptions of Theorem \ref{thm.conv} hold and
let $(u_\eta)_{\eta>0}$ be a family of discrete solutions
to scheme \eqref{sch.ic}-\eqref{2.psigma} constructed in Theorem \ref{thm.ex}.
Then, up to a subsequence, as $\eta\to 0$,
$$
  \na^\eta\bigg(\frac{u_{i,\eta}q(M_\eta)}{p(M_\eta)}\bigg)
	\rightharpoonup \na\bigg(\frac{u_iq(M)}{p(M)}\bigg)
	\quad\mbox{weakly in }L^2(Q_T),
$$
where $u_i$ and $M$ are the limit functions obtained in Proposition \ref{prop.conv}.
\end{lem}

\begin{proof}
This result follows from the proof of~\cite[Lemma 4.4]{CLP03} since
Proposition \ref{prop.conv} guarantees the a.e.\ convergence of
$u_{i,\eta}q(M_\eta)/p(M_\eta)$ to $u_iq(M)/p(M)$.
\end{proof}

Finally, we verify that the limit function $u$ satisfies the Dirichlet
boundary condition in a weak sense.

\begin{lem}[Convergence of the traces]
Let the assumptions of Theorem \ref{thm.conv} hold and
let $(u_\eta)_{\eta>0}$ be a family of discrete solutions
to scheme \eqref{sch.ic}-\eqref{2.psigma} constructed in Theorem \ref{thm.ex}
such that $u_\eta\to u$ and $M_\eta\to M$ a.e.\ in $Q_T$ as $\eta\to 0$. Then
$$
  \frac{u_i q(M)}{p(M)} - \frac{u_i^Dq(M^D)}{p(M^D)}\in L^2(0,T;H_D^1(\Omega)).
$$
\end{lem}

\begin{proof}
Let us define $v_{i,\eta} := u_{i,\eta}q(M_{\eta})/p(M_{\eta})$ for $i=1,\ldots,n$.
Then, using~\cite[Lemma 4.7]{BCH13} and~\cite[Lemma 4.8]{BCH13}, we can prove, thanks to Lemma~\ref{lem.grad} and the $L^{\infty}$-estimate, that up to a subsequence, for all $1 \leq p <+\infty$ as $\eta \rightarrow 0$,
$$
v_{i,\eta} \rightarrow v_i = \frac{u_i \, q(M)}{p(M)} \quad \mbox{strongly in} \, \, L^p(\Gamma^D \times (0,T)), \, \, i=1,\ldots,n,
$$
see for instance the proof of~\cite[Proposition 4.9]{BCH13}. Then, up to a subsequence,
\begin{equation}\label{convtrace}
v_{i,\eta} \rightarrow v_i \quad \mbox{a.e. in} \, \, \Gamma^D \times (0,T), \, \, i=1,\ldots,n.
\end{equation}
Moreover, by construction~\eqref{reconstruc}-\eqref{reconstrucbord},
$$
v_{i,\eta}(x,t) = \frac{u^D_i \, q(M^D)}{p(M^D)} \quad \mbox{for} \, \, (x,t) \in \Gamma^D \times (0,T), \, \, i=1,\ldots,n.
$$
Thus, we deduce from~\eqref{convtrace} that
$$
v_i = \frac{u^D_i \, q(M^D)}{p(M^D)} \quad \mbox{a.e. in} \, \, \Gamma^D \times (0,T), \, \, i=1,\ldots,n,
$$
which concludes the proof.
\end{proof}


\section{Convergence of the scheme}\label{sec.convsch}

We prove in this section that, under the assumptions of Theorem \ref{thm.conv},
the limit function $u=(u_1,\ldots,u_n)$ obtained in Proposition \ref{prop.conv}
is a weak solution to \eqref{1.eq}-\eqref{1.bc}. For this, we follow some
ideas developed in \cite{CCGJ19,CLP03}.

Let $\phi\in C_0^\infty(\Omega \times [0,T))$ and choose 
$\eta=\max\{\Delta x,\Delta t\}$ sufficiently small such that
$\operatorname{supp}(\phi)\subset\{x\in\Omega:\dist(x,\pa\Omega)>\eta\}\times[0,T)$.
In particular, $\phi$ vanishes in any cell $K\in\T$ with $K\cap\pa\Omega\neq\emptyset$.
Again, we abbreviate $\phi_K^k=\phi(x_K,t_k)$ and we fix
$i\in\{1,\ldots,n\}$. Let
\begin{align*}
  & \eps(\eta) = F_{10}^\eta + F_{20}^\eta, \quad\mbox{where} \\
	& F_{10}^\eta = -\int_0^T\int_\Omega u_{i,\eta}\pa_t\phi dxdt
	- \int_\Omega u_{i,\eta}(x,0)\phi(x,0)dx, \\
	& F_{20}^\eta = \int_0^T\int_\Omega p(M_\eta)^2\na^\eta\bigg(
	\frac{u_{i,\eta}q(M_\eta)}{p(M_\eta)}\bigg)\cdot\na\phi dxdt.
\end{align*}
Proposition \ref{prop.conv} and Lemma \ref{lem.convgrad} allow us to perform 
the limit $\eta\to 0$ in these integrals, leading to
\begin{align*}
  \lim_{\eta\to 0}\eps(\eta)
	&= -\int_0^T\int_\Omega u_i\pa_t\phi dxdt - \int_\Omega u_i(x,0)\phi(x,0)dx \\
  &\phantom{xx}{}
	+ \int_0^T\int_\Omega p(M)^2\na\bigg(\frac{u_iq(M)}{p(M)}\bigg)\cdot
	\na\phi dxdt.
\end{align*}
Therefore, it remains to prove that $\eps(\eta)\to 0$ as $\eta\to 0$.

To this end, we multiply \eqref{sch1} by $\Delta t\phi_K^{k-1}$ and sum
over $K\in\T$ and $k=1,\ldots,N_T$, giving
\begin{align*}
  & F_1^\eta + F_2^\eta + F_3^\eta = 0, \quad\mbox{where} \\
	& F_1^\eta = \sum_{k=1}^{N_T}\sum_{K\in\T}\m(K)(u_{i,K}^k-u_{i,K}^{k-1})\phi_K^{k-1}, \\
	& F_2^\eta = \sum_{k=1}^{N_T}\Delta t\sum_{K\in\T}\sum_{\sigma\in\E_{{\rm int},K}}
	\F_{i,K,\sigma}^k\phi_K^{k-1}.
\end{align*}
For the proof of $\eps(\eta)\to 0$ as $\eta\to 0$, it is sufficient to show that
$F_{j0}^\eta-F_j^\eta\to 0$ as $\eta\to 0$ for $j=1,2$.

The arguments in \cite[Section 5.2]{CCGJ19} show that
$$
  |F_{10}^\eta-F_1^\eta| \le CT\m(\Omega)\|\phi\|_{C^1(\overline{Q_T})} \, \eta\to 0
	\quad\mbox{as }\eta\to 0.
$$
The remaining convergence for $|F_{20}^\eta-F_2^\eta|$ is more involved.
First, we rewrite $F_2^\eta$. By the conservation of the numerical fluxes $\F_{i,K,\sigma}+\F_{i,L,\sigma} = 0$ for all the edges $\sigma = K|L \in \E_{{\rm int}}$ and the definition of $\F_{i,K,\sigma}^k$, we infer that
\begin{align*}
  F_2^\eta &= -\sum_{k=1}^{N_T}\Delta t\sum_{K\in\T}
	\sum_{\sigma\in\E_{{\rm int},K}}\F_{i,K,\sigma}^k D_{K,\sigma}\phi^{k-1} \\
	&= \sum_{k=1}^{N_T}\Delta t\sum_{K\in\T}
	p(M_K^k)^2\sum_{\sigma\in\E_{{\rm int},K}}\tau_\sigma
	D_{K,\sigma}\bigg(\frac{u_i^kq(M^k)}{p(M^k)}\bigg) D_{K,\sigma}\phi^{k-1} \\
	&\phantom{xx}+ \sum_{k=1}^{N_T}\Delta t\sum_{K\in\T}
	\sum_{\sigma\in\E_{{\rm int},K}}\tau_\sigma\big((p_\sigma^k)^2-p(M_K^k)^2\big)
	D_{K,\sigma}\bigg(\frac{u_i^kq(M^k)}{p(M^k)}\bigg) D_{K,\sigma}\phi^{k-1} \\
  &=: F_{21}^\eta + F_{22}^\eta.
\end{align*}
Inserting the definition of the discrete gradient $\na^\eta=\na^{\DD^\eta}$,
we can reformulate $F_{20}^\eta$ as 
\begin{align*}
  F_{20}^\eta = \sum_{k=1}^{N_T}\sum_{K\in\T}p(M_K^k)^2
	\sum_{\sigma\in\E_{{\rm int},K}}D_{K,\sigma}\bigg(\frac{u_i^kq(M^k)}{p(M^k)}\bigg)
	\frac{\m(\sigma)}{\m(T_{K,\sigma})}\int_{t_{k-1}}^{t_k}\int_{T_{K,\sigma}}
	\na\phi\cdot\nu_{K,\sigma} dxdt.
\end{align*}
Thus, using the monotonicity of $p$, we have
\begin{align*}
  |F_{20}^\eta-F_{21}^\eta| 
	&\le p(0)^2\sum_{k=1}^{N_T}\sum_{K\in\T}
	\sum_{\sigma\in\E_{{\rm int},K}}\m(\sigma)
	D_\sigma\bigg(\frac{u_i^kq(M^k)}{p(M^k)}\bigg) \\
	&\phantom{xx}{}\times
	\bigg|\int_{t_{k-1}}^{t_k}\bigg(\frac{D_{K,\sigma}\phi^k}{\dist_\sigma}
	- \frac{1}{\m(T_{K,\sigma})}\int_{T_{K,\sigma}}\na\phi\cdot\nu_{K,\sigma}dx
	\bigg)dt\bigg|.
\end{align*}
In view of the proof of Theorem 5.1 in \cite{CLP03}, there exists a constant 
$C_{\rm cons}>0$ such that
\begin{equation*}
  \bigg|\int_{t_{k-1}}^{t_k}\bigg(\frac{D_{K,\sigma}\phi^k}{\dist_\sigma}
	- \frac{1}{\m(T_{K,\sigma})}\int_{T_{K,\sigma}}\na\phi\cdot\nu_{K,\sigma} dx\bigg)
	dt\bigg| \le C_{\rm cons}\Delta t\eta.
\end{equation*}
Applying this inequality and the Cauchy-Schwarz inequality, we obtain
\begin{align*}
	|F_{20}^\eta-F_{21}^\eta| \le p(0)^2C_{\rm cons}\eta
	\bigg(\sum_{k=1}^{N_T}\Delta t\sum_{\sigma\in\E}\m(\sigma)\dist_\sigma\bigg)^{1/2}
	\bigg(\sum_{k=1}^{N_T}\Delta t\bigg|\frac{u_i^kq(M^k)}{p(M^k)}\bigg|_{1,2,\M}^2
	\bigg)^{1/2}.
\end{align*}
It remains to use the mesh regularity \eqref{2.regmesh}, property \eqref{2.para},
and the gradient estimate given by Lemma~\ref{lem.grad} to conclude that, for some constant $C>0$,
\begin{equation}\label{F20F21}
  |F_{20}^\eta-F_{21}^\eta| \le C(\xi,C_3)p(0)^2\eta\to 0 \quad\mbox{as }\eta\to 0.
\end{equation}

We turn to the estimate of $F_{22}^\eta$. To this end, we use the definition of 
$(p_\sigma^k)^2$ to rewrite $F_{22}^\eta$ as $F_{22}^\eta=F_{220}^\eta+F_{221}^\eta$,
where
\begin{align*}
F_{220}^{\eta} &= \sum_{k=1}^{N_T}\Delta t\sum_{K\in\T}
	\sum_{\sigma\in\E_{{\rm int},K}}\tau_\sigma\, \frac{p(M_{K,\sigma}^k)^2-p(M_K^k)^2}{2}\,
	D_{K,\sigma}\bigg(\frac{u_i^kq(M^k)}{p(M^k)}\bigg) D_{K,\sigma}\phi^{k-1} \mathbf{1}_{\{M_K^k>M_{K,\sigma}^k\}},\\
F_{221}^\eta &= \sum_{k=1}^{N_T}\Delta t\sum_{K\in\T}
	\sum_{\sigma\in\E_{{\rm int},K}}\tau_\sigma \frac{p(M_{K,\sigma}^k)^2-p(M_K^k)^2}{2}
	D_{K,\sigma}\bigg(\frac{u_i^kq(M^k)}{p(M^k)}\bigg) D_{K,\sigma}\phi^{k-1}\mathbf{1}_{\{M_K^k\leq M_{K,\sigma}^k\}}.
\end{align*}
It follows from $p(M_K^k)\le p(M_{K,\sigma}^k)$ and the inequality $x^2-y^2\le 2x(x-y)$ 
that
\begin{align*}
  |F_{220}^\eta| &\le 2\eta\|\phi\|_{C^1(\overline{Q_T})}
	\sum_{k=1}^{N_T}\Delta t\sum_{K\in\T}
	\sum_{\sigma\in\E_{{\rm int},K}}\tau_\sigma \\
	&\phantom{xx}{}\times\bigg|\frac{p(M_{K,\sigma}^k)^2-p(M_K^k)^2}{2} \,
	\sqrt{\frac{u_{i,K,\sigma}^kq(M_{K,\sigma}^k)}{p(M_{K,\sigma}^k)}}
	D_{K,\sigma}\bigg(\sqrt{\frac{u_i^kq(M^k)}{p(M^k)}}\bigg)
	\mathbf{1}_{\{M_K^k>M_{K,\sigma}^k\}}\bigg|. 
\end{align*}
A Taylor expansion, for $\widetilde M_\sigma^k=\tilde{\theta}_\sigma M_K^k
+ (1-\tilde{\theta}_\sigma)M_{K,\sigma}^k$ for some $\tilde{\theta}_{\sigma} \in (0,1)$,
$$
  p(M_{K,\sigma}^k)^2-p(M_K^k)^2 = 2p'(\widetilde M_\sigma^k)p(\widetilde M_\sigma^k)
	(M_{K,\sigma}^k-M_K^k),
$$
and the Cauchy-Schwarz inequality give
\begin{align}
  |F_{220}^\eta| &\le 2\eta\|\phi\|_{C^1(\overline{Q_T})}F_{2200}^\eta F_{2201}^\eta,
  \quad\mbox{where} \label{F22} \\
  F_{2200}^\eta &= p(0)\bigg\{\sum_{k=1}^{N_T}\Delta t\sum_{\sigma\in\E}\tau_\sigma
	\bigg(D_\sigma\bigg(\sqrt{\frac{u_i^kq(M^k)}{p(M^k)}}\bigg)
	\bigg)^2\bigg\}^{1/2}, \nonumber \\
  F_{2201}^\eta &= \bigg\{\sum_{k=1}^{N_T}\Delta t\sum_{\sigma\in\E}\tau_\sigma
	p'(\widetilde M_\sigma^k)^2\frac{u_{i,K,\sigma}^kq(M_{K,\sigma}^k)}{p(M_{K,\sigma}^k)}
	(D_\sigma M)^2\mathbf{1}_{\{M_K^k>M_{K,\sigma}^k\}}\bigg\}^{1/2}. \nonumber
\end{align}
Inequality \eqref{4.aux} shows that
$F_{2200}^\eta\le p(0) H(u_\M^0)^{1/2}/p(M^*)$.

For the estimate of $F_{2201}^\eta$, we use $u_{i,K,\sigma}^k\le 1$ and
$C_7:=\sup_{0\le x\le M^*}p'(x)^2/p(x)<\infty$ (this is finite since $M^*<1$)
to infer that
\begin{align*}
  F_{2201}^\eta
	&\le C_7\bigg\{\sum_{k=1}^{N_T}\Delta t\sum_{\sigma\in\E}\tau_\sigma
	q(M_{K,\sigma}^k)(D_\sigma M)^2
	\mathbf{1}_{\{M_K^k>M_{K,\sigma}^k\}}\bigg\}^{1/2} \\
  &= C_7\bigg\{\sum_{k=1}^{N_T}\Delta t\sum_{\sigma\in\E}\tau_\sigma
	(M_{K,\sigma}^k)^{1-a}(1-M_{K,\sigma}^k)^{1+b+\kappa}
	q(M_{K,\sigma}^k) \\
  &\phantom{xx}{}\times\frac{(M_{K,\sigma}^k)^{a-1}}{(1-M_{K,\sigma}^k)^{1+b+\kappa}}
	(D_\sigma M)^2\mathbf{1}_{\{M_K^k>M_{K,\sigma}^k\}}\bigg\}^{1/2}.
\end{align*}
Set $M_\sigma^k = \theta_\sigma M_K^k + (1-\theta_\sigma)M_{K,\sigma}^k$
as in the proof of Lemma \ref{lem.lower}. Using the inequality $(1-M_{K,\sigma}^k)^{1+b+\kappa}\le 1$ together with the monotonicity of 
$x \mapsto x^{a-1}/(1-x)^{-1-b-\kappa}$, we obtain
$$
  F_{2201}^\eta \le C_7\bigg\{\sum_{k=1}^{N_T}\Delta t\sum_{\sigma\in\E}
	\tau_\sigma(M_{K,\sigma}^k)^{1-a}	q(M_{K,\sigma}^k)\frac{(M_{\sigma}^k)^{a-1}}{(1-M_\sigma^k)^{1+b+\kappa}}
	(D_\sigma M)^2\mathbf{1}_{\{M_K^k>M_{K,\sigma}^k\}}\bigg\}^{1/2}.
$$
By \eqref{3.I} and the bound
$$
(M^k_{K,\sigma})^{1-a} \, q(M^k_{K,\sigma}) \leq \frac{M^*}{(a+1) \, p(M^*)^2 \, (1-M^*)^b} \quad \mbox{for all }\sigma \in \E,
$$
this expression is bounded by the entropy production which
is uniformly bounded due to the entropy inequality. We have shown that $F_{2200}^\eta$ and $F_{2201}^\eta$ are bounded uniformly in
$\eta$ such that \eqref{F22} implies that $F_{220}^\eta\to 0$ as $\eta\to 0$.

Now we rewrite $|F_{221}^\eta|$ as
\begin{align*}
  |F_{221}^\eta| &\le 2\eta\|\phi\|_{C^1(\overline{Q_T})}
	\sum_{k=1}^{N_T}\Delta t\sum_{K\in\T}
	\sum_{\sigma\in\E_{{\rm int},K}}\tau_\sigma
	\Bigg|\frac{p(M_K^k)^2-p(M_{K,\sigma}^k)^2}{2} \,
	\sqrt{\frac{u_{i,K}^kq(M_{K}^k)}{p(M_{K}^k)}} \\
	&\phantom{xx}{}\times\Bigg(\sqrt{\frac{u_{i,K}^kq(M^k_K)}{p(M^k_K)}} - \sqrt{\frac{u_{i,K,\sigma}^kq(M^k_{K,\sigma})}{p(M^k_{K,\sigma})}} \Bigg)
	\mathbf{1}_{\{M_K^k\leq M_{K,\sigma}^k\}}\Bigg|. 
\end{align*}
Arguing as for the term $|F_{220}^\eta|$, we see that $F_{221}^\eta \to 0$ 
as $\eta \to 0$.

The previous convergences and \eqref{F20F21} imply that
$$
  |F_{20}^\eta-F_2^\eta| \le |F_{20}^\eta-F_{21}^\eta| + |F_{22}^\eta|\to 0 \quad\mbox{as }\eta\to 0.
$$
To conclude the proof of Theorem~\ref{thm.conv}, it remains to apply \cite[Theorem 2.3]{DMZ18} which shows the uniqueness of the weak solution to~\eqref{1.eq}-\eqref{1.bc} (in the case $\alpha_i=1$ for $i=1,\ldots,n$) and which implies in particular that the whole sequence $(u_\eta)_{\eta > 0}$ converges to the weak solution.


\section{Numerical experiments}
\label{sec.exp.num}

We present some numerical experiments in one and two space dimensions, 
when the biofilm is composed by $n=2$ different species of bacteria
and the function $p$ satisfies hypothesis (H4) (case 1) or not (case 2).

\subsection{Implementation of the scheme}

The finite-volume scheme \eqref{sch.ic}-\eqref{2.psigma} is implemented in MATLAB. 
Since the numerical scheme is implicit in time, one has to solve a nonlinear system of equations at each time step.
In the one-dimensional case, we use a plain Newton method. Starting from 
$u^{k-1}=(u^{k-1}_1, u^{k-1}_2)$, we apply a Newton method with precision 
$\eps = 10^{-10}$ to approximate the solution to the scheme at time step $k$.
In the two-dimensional case, we use a Newton method complemented by an adaptive 
time step strategy to approximate the solution of the scheme at time $k$. 
More precisely, starting again from $u^{k-1}=(u^{k-1}_1, u^{k-1}_2)$, we launch a 
Newton method. Then, if the method did not converge with precision 
$\eps= 10^{-10}$ after at most $50$ steps, we half the time step and restart the 
Newton method. At the beginning of each time step, we double the previous time 
step. Moreover, we impose the condition $10^{-8}\leq \Delta t_{k-1} \leq 10^{-2}$ with 
an initial time step set to $\Delta t_0 = 10^{-5}$.

\subsection{Test case 1}

We introduce a function $p$ that satisfies hypothesis (H4),
\begin{equation}\label{p.1}
  p(x) = \exp(-1/(1-x)) \quad \mbox{for all }x \in [0,1),
\end{equation}
and we choose $a=b=2$. In this case $\kappa=1$ and
$$
  \lim_{M\to 1} (-(1-M)^2) \frac{p'(M)}{p(M)}=1.
$$
This definition of $p$ allows us to compute explicitly the value of $q(M)/p(M)$:
$$
  \frac{q(M)}{p(M)} = \frac{1}{M}\bigg(e^{2/(1-M)}\bigg(M-\frac12\bigg)
  + \frac{e^2}{2}\bigg).
$$

We consider a one-dimensional test case on $\Omega = (0,1)$ with $\Gamma^D = \{0\}$, $\Gamma^N = \{1\}$, $u^D_1 = u^D_2 = 0.1$, and the following initial data:
$$
  u^0_1(x) = u^D_1 + u^D_1 \mathbf{1}_{[0.2,0.5]}(x), \quad 
  u^0_2(x) = u^D_2 + u^D_2 \mathbf{1}_{[0.5,0.8]}(x).
$$
In Figure~\ref{Fig1}, we illustrate the order of convergence in space of the scheme. Since exact solutions to the biofilm model are not explicitly known, we compute a reference solution on a uniform mesh composed of $5120$ cells and with $\Delta t = (1/5120)^2$. We use this rather small value of $\Delta t$
because the Euler discretization in time exhibits a first-order convergence rate, 
while we expect a second-order convergence 
rate in space for scheme \eqref{sch.ic}-\eqref{2.psigma}, due to the 
approximation of $p(M)^2$ in the numerical fluxes. We compute approximate solutions 
on uniform meshes made of respectively $40$, $80$, $160$, $320$, $640$, $1280$, 
and $2560$ cells. Finally, we compute the $L^2$ norm of the difference between the 
approximate solution and the average of the reference solution over $40$, $80$, $160$, $320$, $640$, and $1280$ cells at the final time $T=10^{-3}$. Figure~\ref{Fig1} 
shows the results for $p$ defined in \eqref{p.1} and with different choices of the 
diffusivities $\alpha_1$ and $\alpha_2$. We observe that the scheme converges, even when $\alpha_1 \neq \alpha_2$, with an order around two.

\begin{figure}[!ht]
\begin{center}
\includegraphics[scale=1]{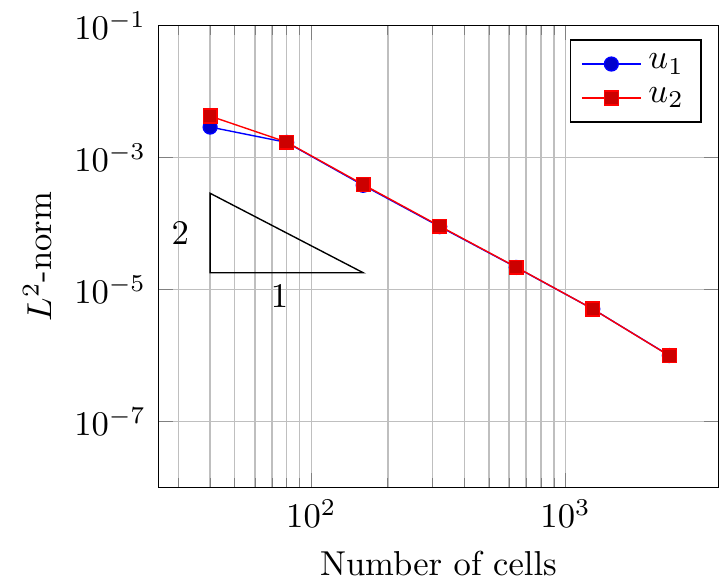}
\includegraphics[scale=1]{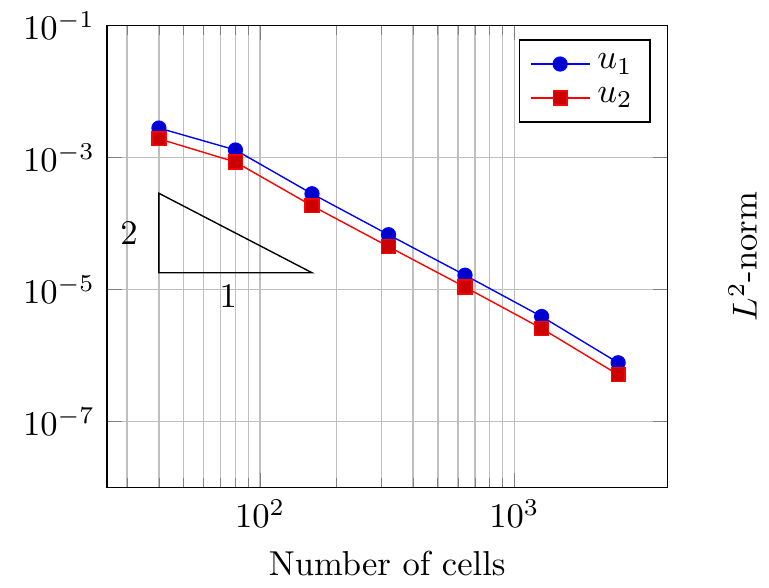}
\caption{$L^2$ norm of the error in space with $\alpha_1 = \alpha_2=1$ (left)
and $\alpha_1=1$ and $\alpha_2 = 10$ (right); $p$ is defined in \eqref{p.1}.}
\label{Fig1}
\end{center}
\end{figure}

Next, we consider a two-dimensional test case on $\Omega = (0,1) \times (0,1)$ 
with $\Gamma^D = \{y=1\}$, $\Gamma^N = \partial \Omega \setminus \Gamma^D$, $u^D_1 = u^D_2 = 0.1$, $\alpha_1=1$, $\alpha_2=5$, and the initial data
$$
  u^0_1(x,y) = u^D_1 + u^D_1 \, \mathbf{1}_{[0.2,0.5]}(x) \mathbf{1}_{[0,0.4]}(y),\quad
  u^0_2(x,y) = u^D_2 + u^D_2 \, \mathbf{1}_{[0.5,0.8]}(x)\mathbf{1}_{[0,0.4]}(y).
$$ 
The mesh of $\Omega = (0,1) \times (0,1)$ is composed of 3584 triangles. 
In Figure~\ref{Fig2}, we show the evolution 
of the biomass $M$ at different times. It is shown in \cite[Theorem 2.2]{DMZ18}
that the steady state is given by $u^{\infty}_1 = u^D_1$ and $u^{\infty}_2 = u^D_2$ 
and that the rate of convergence in the $L^2$ norm is of order $1/t$. 
Figure~\ref{Fig2} (bottom right) 
shows this convergence to the steady state in the $L^2$ norm 
in a semi-logarithmic scale. We remark that the test case used here is close to that 
one used in \cite{DMZ18}. The main difference is the absence of the source term in 
our case. It is worth mentioning that in this case, the rate of convergence of order $1/t$ seems to be sharp, while in~\cite{DMZ18}, the authors observed an exponential 
convergence rate when the source term is given by $u^D_i - u_i$ for $i=1,\ldots,n$.

\begin{figure}[!ht]
\begin{center}
\includegraphics[scale=0.39]{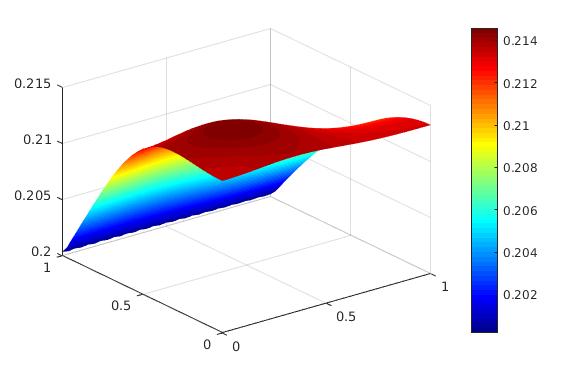}
\includegraphics[scale=0.39]{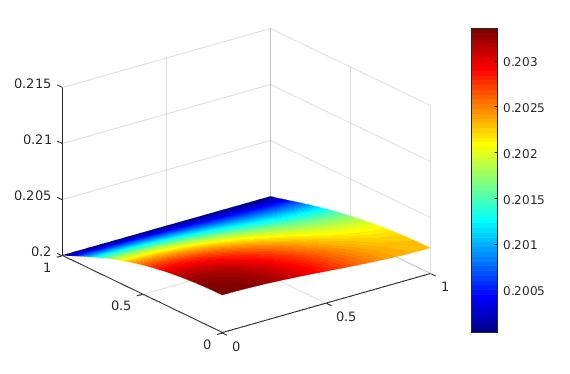}
\includegraphics[scale=0.39]{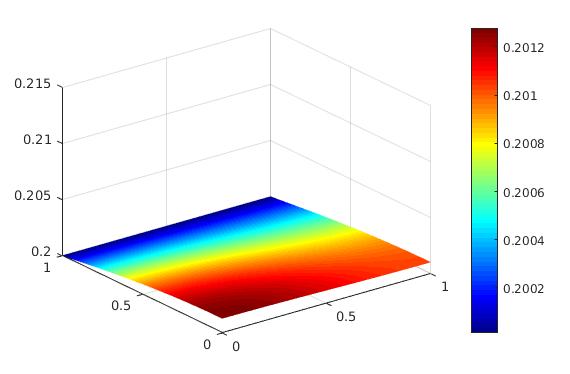}
\includegraphics[scale=0.88]{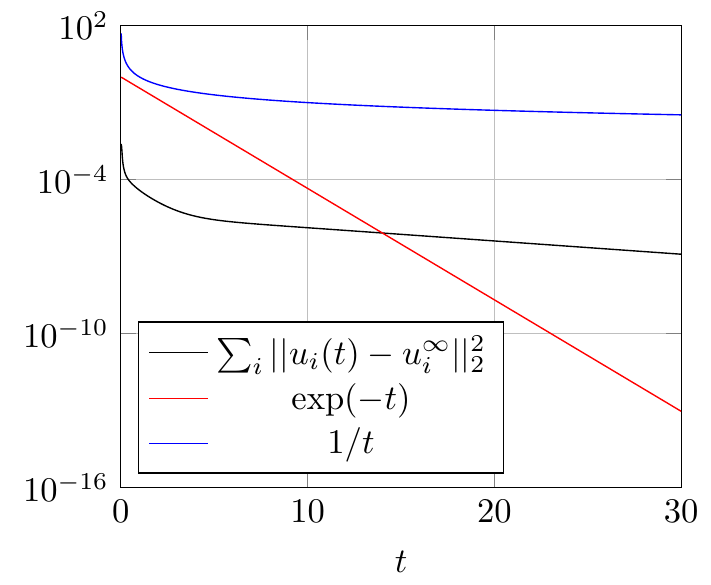}
\end{center}
\caption{Evolution of the biomass $M$ at different times with $p$ defined in \eqref{p.1}. Top left: $t=1$, top right: $t=5$, bottom left: $t=10$.
Bottom right: Convergence of the solutions to the steady states in the $L^2$ norm 
with $p$ defined in \eqref{p.1}.}
\label{Fig2}
\end{figure}



\subsection{Test case 2}

We use a function $p$ that does {\em not} satisfy hypothesis (H4):
\begin{equation}\label{p.2}
  p(x) = 1-x \quad \mbox{for all }x \in [0,1]
\end{equation}
and take $a=b=1$. Also here, we can also compute explicitly $q(M)/p(M)$:
$$
  \frac{q(M)}{p(M)} = \frac{M}{2(1-M)^2}.
$$
As before, we consider first a one-dimensional test case on $\Omega = (0,1)$ 
with $\Gamma^D = \{0\}$, $\Gamma^N = \{1\}$, $u^D_1 = u^D_2 = 0.1$, 
and the initial data
$$
  u^0_1(x) = u^D_1 + u^D_1 \mathbf{1}_{[0.2,0.5]}(x), \quad
  u^0_2(x) = u^D_2 + u^D_2 \mathbf{1}_{[0.5,0.8]}(x).
$$ 
We investigate the $L^2$-convergence rate in space of the scheme for different 
values of $\alpha_1$ and $\alpha_2$; see Figure~\ref{Fig4}. We use the same strategy 
as described in the previous section. In particular, the scheme 
converges with an order around two.

\begin{figure}[!ht]
\begin{center}
\includegraphics[scale=1]{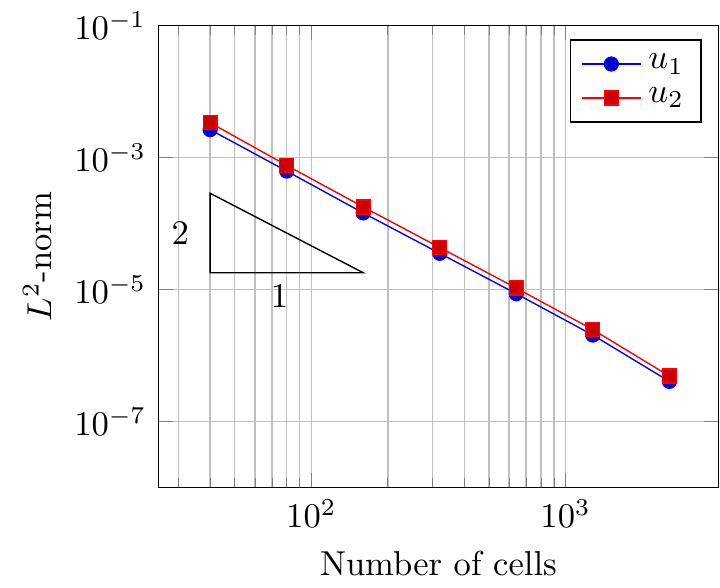}
\includegraphics[scale=1]{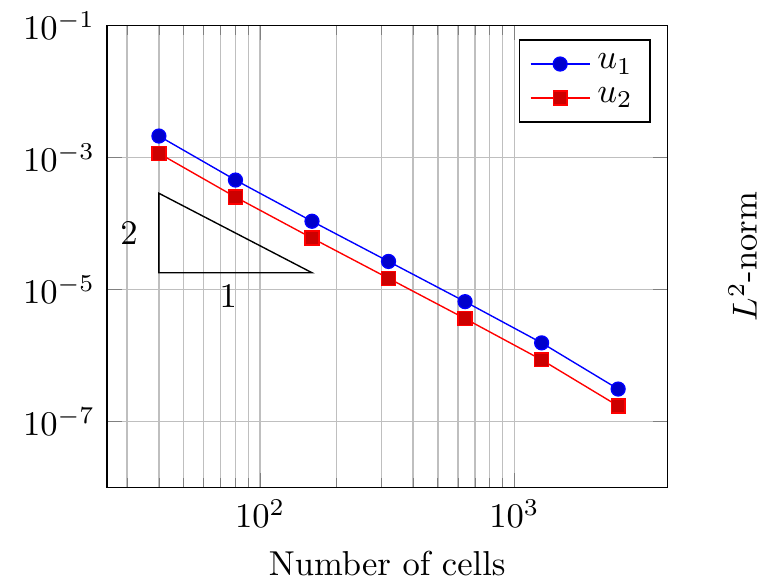}
\caption{$L^2$ norm of the error in space with 
$\alpha_1 = \alpha_2=1$ (left) and $\alpha_1=1$ and $\alpha_2 = 10$ (right); 
$p$ is defined in \eqref{p.2}.}
\label{Fig4}
\end{center}
\end{figure}

Finally, we consider a two-dimensional test case on $\Omega = (0,1) \times (0,1)$ with 
$\Gamma^D = \{y=1\}$, $\Gamma^N = \partial \Omega \setminus \Gamma^D$, 
$u^D_1 = u^D_2 = 0.1$, $\alpha_1=1$, $\alpha_2=5$, and the initial data
$$
  u^0_1(x,y) = u^D_1 + u^D_1 \mathbf{1}_{[0.2,0.5]}(x) \mathbf{1}_{[0,0.4]}(y),\quad
  u^0_2(x,y) = u^D_2 + u^D_2 \mathbf{1}_{[0.5,0.8]}(x)\mathbf{1}_{[0,0.4]}(y).
$$
Again, we choose a mesh of $\Omega = (0,1) \times (0,1)$ consisting of 3584 triangles. 
In Figure~\ref{Fig5}, we show the 
evolution of the biomass $M$ at different times and investigate the rate of
convergence of the solution to the steady state $u^{\infty}_1 = u^D_1$ 
and $u^{\infty}_2 = u^D_2$. We represent the (squared) 
$L^2$ norm of the difference between $u_i$ and $u^\infty_i$ in a semi-logarithmic 
scale with final time $T=30$. Surprisingly, the rate of convergence seems to be
better that the one of order $1/t$ obtained in \cite[Theorem 2.2]{DMZ18}.

\begin{figure}[!ht]
\begin{center}
\includegraphics[scale=0.39]{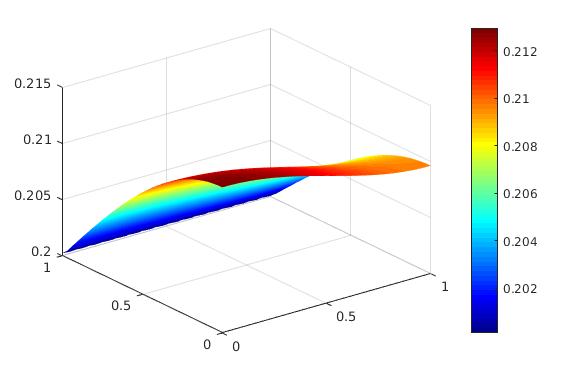}
\includegraphics[scale=0.39]{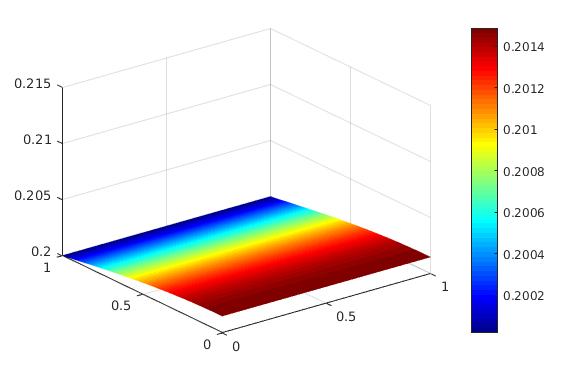}
\includegraphics[scale=0.39]{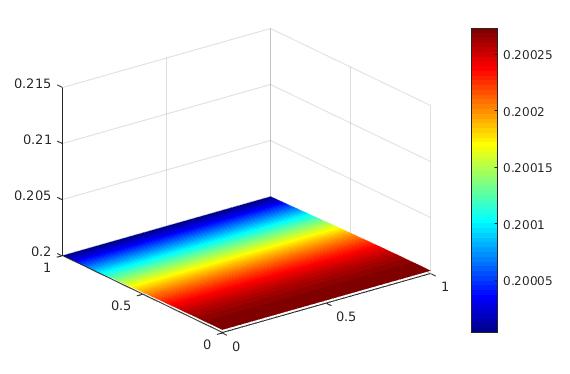}
\includegraphics[scale=0.88]{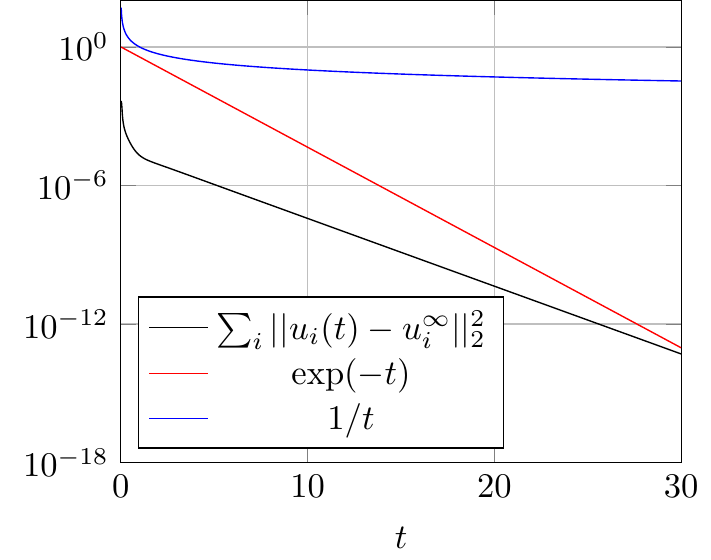}
\end{center}
\caption{Evolution of the biomass $M$ at different times with $p$ defined in
\eqref{p.2}. Top left: $t=1$, top right: $t=5$, bottom: $t=10$.
Bottom right: Convergence of the solutions to the steady states in $L^2$ norm with $p$ defined by~\eqref{p.2}.}
\label{Fig5}
\end{figure}



\end{document}